\documentclass[reqno, 11 pt]{amsart}

\usepackage{amsfonts,amscd,mathrsfs}
\usepackage{amsthm}
\usepackage{amssymb}
\usepackage{latexsym}
\usepackage{url}
\usepackage{cite}
\usepackage{amsmath}
\usepackage{verbatim}
\usepackage{graphicx}
\usepackage{graphics}
\usepackage{color}
\usepackage{epsf}
\usepackage{enumerate}
\usepackage{geometry}
\usepackage{hyperref}
\definecolor{dark-blue}{rgb}{0,0,0.6}
\definecolor{Purple}{rgb}{0.2,0,0.25}
\hypersetup{breaklinks=true,
pagecolor=white,
colorlinks=true,
citecolor=dark-blue,
filecolor=black,%
linkcolor=Purple,%
urlcolor=black
}
\usepackage[nobysame]{amsrefs}

\newtheorem{thm}{Theorem}[section]
\newtheorem{cor}[thm]{Corollary}
\newtheorem{lem}[thm]{Lemma}
\newtheorem{prop}[thm]{Proposition}

\newtheorem{defin}[thm]{Definition}

\theoremstyle{definition}
\newtheorem{expl}[thm]{Example}

\newtheorem{remark}[thm]{Remark}
\newtheorem{alg}[thm]{Algorithm}

\newtheorem{assumption}[thm]{Assumption}

\newcommand{\wt}{\widetilde}

\newcommand{\dom}{\textnormal{dom}}

\newcommand{\Int}{\textnormal{Int}}

\newcommand{\cl}{\overline}

\newcommand{\R}{\mathbb{R}}

\newcommand{\N}{\mathbb{N}}

\newcommand{\argmin}{\textnormal{argmin}}
\newcommand{\sign}{\textnormal{sign}}
\newcommand{\bref}[1]{\textbf{\ref{#1}}} 
\newcommand{\beqref}[1]{\textbf{(\ref{#1})}} 

\subjclass[2010]{90C25, 49M27, 90C31, 47J25, 90C30, 54C30, 26B25}
\keywords{Bregman divergence, Lipschitz continuity, minimization, TEPROG, telescopic proximal gradient method, strongly convex}

\begin{document}
\date{March 19, 2019}
\title[A Telescoping Bregmanian Proximal Gradient Method]{A Telescoping Bregmanian Proximal Gradient Method  Without the Global Lipschitz Continuity Assumption}
\author{Daniel Reem}
\address{Daniel Reem, Department of Mathematics, The Technion - Israel Institute of Technology, 3200003 Haifa, Israel.} 
\email{dream@technion.ac.il}
\author{Simeon Reich}
\address{Simeon Reich, Department of Mathematics, The Technion - Israel Institute of Technology, 3200003 Haifa, Israel.} 
\email{sreich@technion.ac.il}
\author{Alvaro De Pierro}
\address{Alvaro De Pierro, CNP$\textnormal{q}$, Brazil}
\email{depierro.alvaro@gmail.com}

\maketitle

\begin{abstract}
The problem of minimization of the sum of two convex functions  
has various theoretical and real-world applications. One of 
the popular methods for solving this problem is the proximal 
gradient method (proximal forward-backward algorithm). 
A very common assumption in the use of this method is that 
the gradient of the smooth term is globally Lipschitz continuous. However, this assumption 
is not always satisfied in practice, thus casting a limitation 
on the method. In this paper, we discuss, in a wide class of 
finite and infinite-dimensional spaces, a new variant 
of the proximal gradient method  which does not impose the 
above-mentioned  global Lipschitz continuity assumption. 
A key contribution of the method is the dependence of the iterative steps 
on a certain telescopic decomposition of the constraint set into subsets. 
Moreover, we use a Bregman divergence in the proximal 
forward-backward operation. Under certain practical 
conditions, a non-asymptotic rate of convergence (that is, 
in the function values) is established, as well as the weak 
convergence of the whole sequence to a minimizer. We also 
obtain a few auxiliary results of independent interest. 
\end{abstract}

\section{Introduction}\label{sec:Intro}\subsection{Background} 
The problem of minimization of the sum of two convex functions appears in various 
areas in science and technology, including machine learning, inverse problems, image processing 
and signal processing \cite{BachJenattonMairalObozinski2012jour}, \cite{BeckTeboulle2009jour}, \cite{ BerteroBoccacciDesideraVicidomini2009jour}, \cite{CombettesWajs2005jour}, \cite{De-MolDe-VitoRosasco2009jour}, \cite{FigueiredoBioucas-DiasNowak2007jour}, \cite{ParikhBoyd2014jour}, \cite{Tseng2010jour}. Here, we are given an objective function, which is the sum of two convex functions, one of them is smooth (and may vanish identically), but the other may not be differentiable. The goal is to find the minimal value of the objective 
function over the constraint set, namely over some given constraint subset of the ambient space, and possibly also to find a minimizer (if it exists). 

One of the most popular methods for solving 
this problem is the proximal gradient method (also called the ``forward-backward algorithm'' or the ``proximal point algorithm''). Initial versions of this  algorithm, in various settings and forms,  were studied by Martinet \cite{Martinet1970jour}, Rockafellar \cite{Rockafellar1976jour}, Bruck and Reich \cite{BruckReich1977jour}, Passty \cite{Passty1979jour}, Br\'ezis and Lions \cite{BrezisLions1978jour}, and Nevanlinna and Reich \cite{NevanlinnaReich1979jour}, and since then many more developments have occurred and  many more authors have been involved in the investigation of this algorithm. In its very basic form, this algorithm produces iterations,  which are obtained by minimizing, in the iterative step, the sum of the objective function and a quadratic term. 

A generalization of this basic form calls for replacing the quadratic term by a Bregman divergence (Bregman distance, Bregman measure of distance) and for replacing the objective function by an approximation of it. The Bregman  divergence is a certain  substitute for a distance, induced by a given well-chosen function. It  has found various applications, among them in  optimization, nonlinear analysis, inverse problems, machine learning and computational geometry.

A very common  assumption (implicit or explicit) in the use of the various forms (Bregmanian or not) of the proximal gradient algorithm is that the gradient of the smooth term in the objective function is globally Lipschitz continuous. While this assumption holds in several important cases, it is not always satisfied (for example, in some instances of image processing  algorithms \cite[Subsection 5.1]{BauschkeBolteTeboulle2017jour}, \cite[p. 9]{BerteroBoccacciDesideraVicidomini2009jour}, \cite[pp. 1063, 1065]{MarkhamConchello2001jour}, \cite[p. 262]{DeyEtal2006jour}; see also Section \bref{sec:ModelProblem} and Example \bref{ex:ell_p-ell_1} below). This lack of global Lipschitz continuity casts a limitation on the method. As a matter of fact, to the best of our knowledge, only very few papers \cite{BauschkeBolteTeboulle2017jour}, \cite{Bello-CruzNghia2016jour}, \cite{BolteSabachTeboulleVaisbourd2018jour}, \cite{Cohen1980jour}, \cite{Nguyen2017jour} do not impose this global Lipschitz condition (more details about these papers and related works can be found in Section \bref{sec:FurtherComparison}). 

\subsection{Contributions:}\label{subsec:Contributions}  In this paper, we discuss, in a general setting,  new variants of the proximal gradient method  which do  not require the above-mentioned  global Lipschitz continuity assumption.  We consider a broad class of finite- and infinite-dimensional spaces  (real reflexive Banach spaces), allow constrained minimization and backtracking, and use a Bregman divergence in the  iterative step instead of just a quadratic term (this gives more flexibility to the users). The method, assumptions and results discussed in our paper are considerably different from any other relevant work (see  Section \bref{sec:FurtherComparison}). 

A key and novel contribution of our method is the dependence of the iterations on  certain well-chosen subsets which form a telescopic decomposition of the constraint set. More precisely, the forward-backward operation in a given iteration is performed over a well-chosen subset of the constraint set and not over the entire constraint set;  by ``telescopic  decomposition'' we mean that the well-chosen subset in a given iteration is contained in the corresponding subset of the next iteration and the union of all of these subsets is the entire constraint set. Hence we regard our method as being a ``telescopic proximal gradient method'', call it TEPROG and regard the above-mentioned family of well-chosen sets as the ``telescopic  sequence''. 

A major advantage of the above-mentioned dependence of the iterative steps  on well-chosen subsets is that it gives the users a lot of flexibility and, in particular, allows them to assume that the gradient of the smooth term is Lipschitz continuous merely on the above-mentioned subsets (an assumption which frequently holds since these subsets are often bounded), rather than globally Lipschitz continuous.  Based on results which have  recently been established  in \cite{ReemReichDePierro2019jour(BregmanEntropy)},  we obtain, under certain practical conditions, explicit non-asymptotic 
rates of convergence to the optimal value (that is, convergence of the function values to the optimal value), as well as the weak convergence of the whole sequence to a minimizer. In fact, we show that sometimes sublinear  non-asymptotic rates of convergence can be obtained, and sometimes these non-asymptotic rates can be arbitrarily close to the sublinear rate. 

Another contribution of our paper is a few auxiliary results which seem to be of independent interest. One of them is Lemma \bref{lem:FB_k} which generalizes a key result in  \cite{BeckTeboulle2009jour} (namely  \cite[Lemma 2.3]{BeckTeboulle2009jour}).  A second result is Lemma \bref{lem:zInSU} which presents sufficient conditions for the minimizer which appears in the proximal operation to be an interior point.

\subsection{Paper layout:} In Section \bref{sec:FurtherComparison}, we further compare our method and results to several relevant ones from the literature. In Section \bref{sec:ModelProblem}, we illustrate the general minimization problem that  we solve in this paper using a useful model problem. In Section \bref{sec:Preliminaries}, we introduce some basic notation and definitions and also recall well-known facts. In Section \bref{sec:ISTA-Bregman}, we present our telescopic proximal method. In Section \bref{sec:Convergence}, we present the main convergence results. In Section \bref{sec:Examples}, we present several relevant examples. The proofs of the convergence results, as well as relevant auxiliary assertions, are presented in Section \bref{sec:Proofs}. We conclude the paper in Section \bref{sec:Conclusions}. Section \bref{sec:Appendix} is a short appendix which contains the proofs of some assertions.

\section{Further Comparison to the Literature}\label{sec:FurtherComparison} 
As far as we know, our method and results are new. In particular, as far as we know the idea regarding the telescopic sequence of sets on which the gradient of the smooth term is Lipschitz continuous (and not necessarily Lipschitz continuous on the entire constraint set) is novel. However, as mentioned in Section \bref{sec:Intro}, there are several works which are somewhat related to our paper. In this section we present more details about these works and elaborate more on some of the main differences between them and our own. 

The relevant works are \cite{BauschkeBolteTeboulle2017jour}, \cite{BeckTeboulle2009jour}, \cite{Bello-CruzNghia2016jour}, \cite{BolteSabachTeboulleVaisbourd2018jour}, \cite{Cohen1980jour}, \cite{Nguyen2017jour}, \cite{Tseng2008prep}. In a nutshell, in most of them (with the exception of \cite{Cohen1980jour}) the goal is to minimize an objective function $F=f+g$. In addition, in most of these works (with the exception of \cite{Bello-CruzNghia2016jour}, \cite{Cohen1980jour}, \cite{Nguyen2017jour}) the setting is a finite-dimensional Euclidean space. With the exception of \cite{Tseng2008prep}, in all of these works the proximal operations do not depend on a special sequence of sets, and only in \cite{BauschkeBolteTeboulle2017jour}, 
\cite{BolteSabachTeboulleVaisbourd2018jour}, \cite{Nguyen2017jour}, \cite{Tseng2008prep} these proximal operations depend on a Bregman divergence. With the exception of \cite{BolteSabachTeboulleVaisbourd2018jour}, both $f$ and $g$ are assumed to be convex.  Backtracking is discussed only in  \cite{BeckTeboulle2009jour}. 

Here are a few more details. In ISTA \cite[Sections 2-3]{BeckTeboulle2009jour}, the  objective function $F=f+g$ is defined on the whole space (no constrained minimization); in addition, the gradient of the smooth term $f$ is assumed to be globally Lipschitz continuous; it is shown that under these conditions the sequence converges non-asymptotically at a sublinear rate, namely, $O(1/k)$. For a predecessor of ISTA, see \cite[Theorem 3.4]{CombettesWajs2005jour}.   

In  \cite{Cohen1980jour} the setting is a real reflexive Banach space which is sometimes assumed to be a real Hilbert space; each iterative step (in \cite[Algorithm 2.1]{Cohen1980jour}) is based on a minimization of the sum of three terms: a non-smooth term, a certain convex functional (depending on the iteration), and a certain linear term; the convex functional should have a Lipschitz continuous gradient on the constraint set and this gradient should also be strongly monotone; moreover, all of the Lipschitz constants should be uniformly bounded from above and the strong monotonicity constants should be bounded away from 0; under these and additional assumptions, among them that the gradient of the smooth term is  Lipschitz continuous on a  constraint subset, the main theorem \cite[Theorem 2.1]{Cohen1980jour} shows that each weak cluster point solves the minimization problem, and that non-asymptotic convergence holds; under further assumptions a strong convergence result is established; no rate of convergence of any kind is established. There are other theorems and algorithms in \cite{Cohen1980jour}, but they are closely related to the ones mentioned above and the main differences between them and our paper still hold.

In the 2008 preprint \cite{Tseng2008prep}, the space $X$ is a real normed  space which is probably either finite-dimensional (there are many similarities to \cite{Tseng2010jour} where $X$ is finite-dimensional) or at least a reflexive Banach space, since otherwise it is not clear why the iterations in various places, for instance in  \cite[Equations (8),(12),(17),(28),(31),(45),(47),(48)]{Tseng2008prep}, are well defined. The paper \cite{Tseng2008prep} is the only relevant paper of which we are aware, where there are subsets (closed and convex) $S_k$ on which the iterations depend, but they are not almost arbitrary as in our method (see Algorithm \bref{alg:LipschitzStep} and Algorithm \bref{alg:BacktrackingStep}): rather, they are either the whole space or they have a certain complicated form (which seems to be inspired by \cite[p. 240]{Nemirovski2004jour}) in order to make sure that each of them contains an optimal solution; see \cite[Equations (16), (46)]{Tseng2008prep}; on the other hand, they are not assumed to satisfy $\cup_{k=1}^{\infty}S_k=C$ and  $S_k\subseteq S_{k+1}$ for all $k$  as in our case (Assumption \bref{assum:MonotoneStronglyConvex} below): actually, at least in \cite[p. 8, Algorithm 3]{Tseng2008prep}, they are required to satisfy $S_k\supseteq S_{k+1}$ for all $k$; in addition, \cite{Tseng2008prep} assumes that $f'$ is globally Lipschitz continuous (that is, on $C$, which is actually the whole space in \cite{Tseng2008prep}) and the Bregman function $b$ is assumed to be globally strongly convex with a strong convexity parameter which is equal to 1. 

In \cite{Bello-CruzNghia2016jour} the setting is a Hilbert space; the iterative schemes suggested there are based on a line search strategy; under certain assumptions, an $O(1/k)$ non-asymptotic rate of convergence is established, and under further assumptions (finite-dimensionality), an $o(1/k)$ non-asymptotic rate of convergence is proven.   

In \cite{Nguyen2017jour} the setting is a real reflexive Banach space and the iterative step is based on a Bregman divergence which is induced by a Legendre function (which should satisfy additional conditions), on a sequence of other Legendre functions and on sequences of positive parameters; a key assumption imposed there is a certain ``descent condition'': in the notation of our paper (which is different from  \cite{Nguyen2017jour}), it means that if $B_b$ denotes the Bregman divergence induced by $b$ and $B_f$ denotes the Bregman divergence induced by the smooth term $f$, then it is assumed that there exists some $\beta>0$ such that $B_b(x,y)\geq \beta B_f(x,y)$ for all $x\in \dom(b)$ and $y\in \Int(\dom(b))$; under this condition and additional ones, a weak convergence result is established, and under further (and stronger) assumptions,  a strong convergence result is also established; no rate of convergence (asymptotic or non-asymptotic) is established. 

In \cite{BauschkeBolteTeboulle2017jour} the space is a finite-dimensional Euclidean space and the iterative step is based on a Bregman divergence and on a sequence of positive parameters;  the Bregman function is well chosen: for example, it should be Legendre and satisfy a descent condition similar to \cite{Nguyen2017jour} (called ``descent lemma'' in \cite[Condition (LC) and Lemma 1]{BauschkeBolteTeboulle2017jour}; it seems, however, that this ``descent condition'' has independently been obtained  in both works at around the same time); under these and additional conditions, an $O(1/k)$ non-asymptotic rate of convergence is established; under  additional assumptions the convergence of the iterative scheme to a solution of the problem is also  established.

The recent paper \cite{BolteSabachTeboulleVaisbourd2018jour} (see  also some of the references therein) does not assume that $f'$ is globally Lipschitz continuous and does not assume that $f$ or $g$ are convex;   since \cite{BolteSabachTeboulleVaisbourd2018jour} is devoted to nonconvex optimization, we do not further elaborate on it, with the exception of saying that our method, the setting that we consider here, and our results are significantly different from the method, setting and results of \cite{BolteSabachTeboulleVaisbourd2018jour}, even when in \cite{BolteSabachTeboulleVaisbourd2018jour} one restricts one's attention to the convex case. 

Finally, we note that despite the differences in the method, settings and assumptions between our paper and other papers, there are some similarities. For instance, the proof of our main convergence theorem (Theorem \bref{thm:BISTA-Convergence} below) is partly inspired by \cite[Section 3]{BeckTeboulle2009jour} and also has some similarities to \cite[Proofs of Theorem 1 and 2]{BolteSabachTeboulleVaisbourd2018jour} (nevertheless,  there are still differences between our proofs and other proofs, for instance because we use the limiting difference property of the Bregman divergence, while in the above-mentioned works this is not done).

\section{A Model Problem}\label{sec:ModelProblem}
In this short section, we illustrate the optimization problem that we intend to solve by using a useful example involving linear inverse problems. Let $n\in\N$ and $p\in [2,\infty[$ be given. Assume that $0\neq A:\R^n\to\R^m$ is a given linear operator and $c\in \R^m$ is a given vector, where $m\in\N$ is given. Denote $\|y\|_p:=\left(\sum_{i=1}^m|y_i|^p\right)^{1/p}$ for all $y=(y_i)_{i=1}^m\in \R^m$. Fix some $\lambda>0$ (a regularization parameter) and denote $f(x):=\frac{1}{p}\|Ax-c\|_p^p$, $g(x):=\lambda\|x\|_1$ and $F:=f+g$ for each $x\in C:=\R^n$. The minimization problem is to estimate  
\begin{equation}\label{eq:ell_p-ell_1}
\inf\left\{\frac{1}{p}\|Ax-c\|_p^p+\lambda\|x\|_1:\,\,x\in \R^n\right\}.
\end{equation}
In the particular case where $p=2$, this is the familiar $\ell_2-\ell_1$ minimization problem which is popular in machine learning, compressed sensing, and signal/image processing \cite{BachJenattonMairalObozinski2012jour}, \cite{BeckTeboulle2009jour}, \cite{ParikhBoyd2014jour}, but the choice $p=2$ is somewhat arbitrary and it seems that it is driven mostly by convenience and some a priori statistical assumptions on the data which often do not hold. Perhaps another reason for the popularity of the choice $p=2$ is because in this case $f'$ is globally Lipschitz continuous, while this is no longer true when $p>2$.  Nevertheless, we show in Example \bref{ex:ell_p-ell_1} below, in a rather  detailed manner, how one can  apply TEPROG in order to solve \beqref{eq:ell_p-ell_1} even when $p>2$. TEPROG and a similar analysis can be used in closely related scenarios, for instance when the smooth term in \beqref{eq:ell_p-ell_1} is a different proximity function, such as $KL(c,Ax)$ or $KL(Ax,c)$ (under some non-negativity assumptions on $c$, $x$ and $A$), where $KL(w,z):=\sum_{j=1}^m \left[w_j\log(w_j/z_j)-w_j+z_j\right]$, $w\in [0,\infty[^m$, $z\in ]0,\infty[^m$ is the Kullback-Leibler divergence (see \cite[Subsection 5.1]{BauschkeBolteTeboulle2017jour} and the references therein for a related discussion).

\section{Notation and Definitions}\label{sec:Preliminaries}
This section introduces a few basic definitions, which are used later in the paper. 

\subsection{Basic Notation and Assumptions:} Unless otherwise stated, we  consider a real normed space $(X,\|\cdot\|)$, $X\neq\{0\}$, which, frequently, will be explicitly  assumed to be a real reflexive Banach space. Along the paper we use well-known and standard notions and notations from convex analysis, such as the subdifferential of a function from $X$ to $]-\infty,\infty]$, the effective domain of such a function, the Fr\'echet and G\^ateaux derivatives of the function (whenever they exist) and so on. See, for instance, \cite{BauschkeCombettes2017book}, \cite{Rockafellar1970book}, \cite{VanTiel1984book}, \cite{Zalinescu2002book} for some relevant sources. In particular, we let $\langle x^*,x\rangle:=x^*(x)$ for each $x^*$ in the dual $X^*$ of $X$ and each $x\in X$. We denote by $]s,t[$, $[s,t[$, $]s,t]$ and $[s,t]$ the open, right-open, left-open and closed interval, respectively, having $s$  as its left-endpoint and $t$ as its right-endpoint, where  $-\infty\leq s\leq t\leq\infty$.  

A function $h:S\subseteq X\to X^*$  is said to be weak-to-weak$^*$ sequentially continuous at $x\in S$ if for each sequence $(x_i)_{i=1}^{\infty}$ in $S$ which converges weakly to $x$ and for each $w\in X$, we have $\lim_{i\to\infty}\langle h(x_i),w\rangle=\langle h(x),w\rangle$. If $h$ is weak-to-weak$^*$ sequentially continuous at each $x\in S$, then $h$ is said to be weak-to-weak$^*$ sequentially continuous on $S$. Of course, when $X$ is reflexive, then saying that $h$ is weak-to-weak$^*$ sequential continuous is the same as  saying that it is weak-to-weak sequential continuous, or, briefly, that it is weakly sequentially continuous. We say that $h$ is coercive on $S$ if $\lim_{\|x\|\to\infty, x\in S}h(x)=\infty$ and supercoercive on $S$ if  $\lim_{\|x\|\to\infty, x\in S}h(x)/\|x\|=\infty$. We say that $b:X\to ]-\infty,\infty]$ with $U:=\Int(\dom(b))$ (namely, $U$ is the interior of the effective domain of $b$) is essentially smooth on $X$, or on $U$, if $U\neq\emptyset$ and $b$ is proper, convex and G\^ateaux differentiable on $U$,  and also $\lim_{i\to\infty}\|b'(x_i)\|=\infty$ for every sequence $(x_i)_{i=1}^{\infty}$ in $U$ which converges to a boundary point of $U$, where $b'$ is the gradient of $b$.

\subsection{Relative Uniform Convexity}
The next definition presents the central concepts of uniform convexity, relative uniform convexity and (relative) strong convexity. 

\begin{defin}\label{def:TypesOfConvexity}
Assume that $b:X\to]-\infty,\infty]$ is convex and proper. Suppose that $S_1$ and $S_2$ are two nonempty subsets (not necessarily convex) of $\dom(b)$.  
\begin{enumerate}[(I)] 
\item\label{def:RelativeUniformlyConvex} The function $b$ is called uniformly convex relative to $(S_1,S_2)$ (or relatively uniformly convex on  $(S_1,S_2)$) if there exists $\psi:[0,\infty[\to [0,\infty]$, called a relative gauge, such that $\psi(t)\in ]0,\infty]$ whenever $t>0$ and for each $\lambda\in ]0,1[$ and each $(x,y)\in S_1\times S_2$, 
\begin{equation}\label{eq:RelativeUniformConvexity}
b(\lambda x+(1-\lambda)y)+\lambda(1-\lambda)\psi(\|x-y\|)\leq \lambda b(x)+(1-\lambda) b(y).
\end{equation}
If $S:=S_1=S_2$ and $b$ is uniformly convex relative to $(S_1,S_2)$, then $b$ is said to be  uniformly convex on $S$. 
\item\label{def:ModulusConvexity} The optimal gauge of $b$ relative to $(S_1,S_2)$  is the function defined for each $t\in [0,\infty[$ by:
\begin{multline}\label{eq:Relative_psi}
\psi_{b,S_1,S_2}(t):=\inf\left\{ \frac{\lambda b(x)+(1-\lambda)b(y)-b(\lambda x+(1-\lambda)y)}{\lambda(1-\lambda)}: (x,y)\in S_1\times S_2, 
\right.\\\left.
\|x-y\|=t,\lambda\in]0,1[
\vphantom{\frac{\lambda b(x)+(1-\lambda)b(y)-b(\lambda x+(1-\lambda)y)}{\lambda(1-\lambda)}}
\right\},
\end{multline} 
where we use the standard convention that $\inf \emptyset:=\infty$, namely, if there does not exist $(x,y)\in S_1\times S_2$ such that $\|x-y\|=t$, then $\psi_{b,S_1,S_2}(t):=\infty$. 
 The optimal gauge is also called the modulus of relative uniform convexity of $b$ on $(S_1,S_2)$ or simply the optimal relative gauge, and $b$ is uniformly convex on $(S_1,S_2)$ if and only if $\psi_{b,S_1,S_2}(t)>0$ for every $t\in ]0,\infty[$. If $S:=S_1=S_2$, then we denote $\psi_{b,S}:=\psi_{b,S_1,S_2}$ and call $\psi_{b,S}$ the modulus of uniform convexity of $b$ on $S$.
\item The function $b$ is said to be uniformly convex on closed, convex and bounded subsets of $\dom(b)$ if $b$ is  uniformly convex on each nonempty subset $S\subseteq \dom(b)$ which is closed, convex and bounded.  
\item\label{def:StronglyConvexGlobal} The function $b$ is said to be strongly convex relative to $(S_1,S_2)$ if there exists $\mu>0$ (which depends on $S_1$ and $S_2$), called a parameter of strong convexity  of $b$ on $(S_1,S_2)$, such that $b$ is uniformly convex relative to $(S_1,S_2)$ with $\psi(t):=\frac{1}{2}\mu t^2$, $t\in [0,\infty[$, as a relative gauge. If $S:=S_1=S_2$ and $b$ is strongly convex relative to $(S_1,S_2)$, then $b$ is said to be strongly convex on $S$, namely 
for each $\lambda\in ]0,1[$ and each $x,y\in S$, we have 
\begin{equation}\label{eq:StrongConvexity}
b(\lambda x+(1-\lambda)y)\leq \lambda b(x)+(1-\lambda) b(y)-\frac{1}{2}\mu\lambda(1-\lambda)\|x-y\|^2.  
\end{equation}
\end{enumerate}
\end{defin}

\begin{remark}\label{rem:StronglyConvex}$\,$
The notion of relative uniform convexity was introduced and investigated in \cite{ReemReichDePierro2019jour(BregmanEntropy)}, where various examples, results and other relevant details can be found.  Of course, uniformly convex and strongly convex functions (not relatively uniformly convex) are well-known concepts: see, for instance,  \cite[pp. 63-66]{Nesterov2004book} and \cite[pp. 203--221]{Zalinescu2002book}.  
\end{remark}

\subsection{Bregman Divergences}\label{subsec:Bregman}
We now discuss (semi) Bregman functions and divergences.  
\begin{defin}\label{def:BregmanDiv}
Suppose that $b:X\to ]-\infty,\infty]$. Let $\emptyset\neq U\subseteq X$. 
\begin{enumerate} 
\item We say that $b$ is a {\bf semi-Bregman function}  with respect to $U$ ({\bf the zone of $b$}) if the following conditions  hold:
\begin{enumerate}[(i)]
\item\label{BregmanDef:U}  $U=\Int(\dom(b))$ (in particular, $\Int(\dom(b))\neq \emptyset$) and $b$ is G\^ateaux differentiable in $U$. 
\item\label{BregmanDef:ConvexLSC} $b$ is convex and lower semicontinuous on $X$ and strictly convex on $\dom(b)$.
\end{enumerate} 
\item We say that $B$ is the {\bf Bregman divergence} (or {\bf the Bregman distance}, or {\bf the Bregman measure of distance}) associated with $b$, if $B$ is defined by 
\begin{equation}\label{eq:BregmanDistance}
B(x,y):=\left\{\begin{array}{lll}
b(x)-b(y)-\langle b'(y),x-y\rangle, & \forall (x,y)\in \dom(b)\times \Int(\dom(b)),\\
\infty & \textnormal{otherwise}.
\end{array}\right.
\end{equation}
\item We say that $b$ (or $B$) has the {\bf limiting difference property} if for each $x\in \dom(b)$ and each weakly convergent sequence $(y_i)_{i=1}^{\infty}$ in $U$, if the weak limit of $(y_i)_{i=1}^{\infty}$ is some $y\in U$, then $B(x,y)=\lim_{i\to\infty}(B(x,y_i)-B(y,y_i))$. 
\item We say that $B$ has bounded level-sets of the first type if for each $\gamma\in [0,\infty[$ and  each $x\in \dom(b)$, the level-set $L_1(x,\gamma):=\{y\in U: B(x,y)\leq \gamma\}$ is  bounded. 
\end{enumerate}
\end{defin}

Here are a few comments regarding Definition \bref{def:BregmanDiv}.
\begin{remark}\label{rem:BregmanDivergence}
\begin{enumerate}[(i)]
\item The Bregman divergence is, of course, a classical notion. It was  introduced  by Bregman \cite{Bregman1967jour}  in 1967 and since then it has found applications in various fields, among them optimization, nonlinear analysis, inverse problems, machine learning, and computational geometry. This divergence is not a true metric (for example, because it does not satisfy the triangle inequality), but it still enjoys various properties which make it a useful substitute for a distance (for example, $B(x,y)\geq 0$ for all $x\in \dom(b)$ and $y\in \Int(\dom(b))$ and $B(x,y)=0$ if and only if $x=y\in U$: see, for instance, \cite[Proposition 4.13(III)]{ReemReichDePierro2019jour(BregmanEntropy)}). Many more details about this notion, including historical details, a long list of relevant references, various mathematical properties, and a thorough re-examination,  can be found in  the recent paper \cite{ReemReichDePierro2019jour(BregmanEntropy)}. 

\item A semi-Bregman function generalizes the notion of ``a Bregman function'' (the term ``semi-Bregman function'' seems to be new, although it has been used here and there without this explicit name). In finite-dimensional Euclidean spaces, Bregman functions should satisfy more conditions in addition to the ones mentioned in Definition \bref{def:BregmanDiv}. See, for example, \cite[Definition 2.1]{CensorLent1981jour}, \cite[Definition 2.1]{CensorReich1996jour},  \cite[Definition 2.1]{DePierroIusem1986jour} (in this finite-dimensional setting, the limiting difference property is just a consequence of the classical definition and it is not mentioned explicitly).

\item\label{item:LimitingDifferenceProperty} A sufficient condition which ensures that $b$ has the limiting difference property is that $b'$ is weak-to-weak$^*$ sequentially continuous: see \cite[Proposition 4.13(XIX)]{ReemReichDePierro2019jour(BregmanEntropy)}. Of course, if $X$ is finite-dimensional, then $b'$ is continuous (this is a consequence of \cite[Corollary 25.5.1, p. 246]{Rockafellar1970book}) and hence  weak-to-weak$^*$ sequentially continuous; see \cite[Proposition 5.6 and Remark 5.7]{ReemReichDePierro2019jour(BregmanEntropy)} for sufficient conditions which ensure that $b'$ is weak-to-weak$^*$ sequentially continuous in infinite-dimensional settings. 

\item\label{item:LevelSetBounded} An immediate sufficient condition which ensures that all the first type level-sets of $B$ will be bounded is that $U$ is bounded. A less immediate such sufficient condition is that  for each $x\in\dom(b)$, there exists $r_x\geq 0$  such that the subset $\{w\in U: \|w\|\geq r_x\}$ is nonempty and $b$ is uniformly convex relative to $(\{x\},\{w\in U: \|w\|\geq r_x\})$ with a gauge $\psi_x$ which  satisfies $\lim_{t\to\infty}\psi_x(t)=\infty$: see \cite[Proposition 4.13(XV)]{ReemReichDePierro2019jour(BregmanEntropy)}. This latter condition holds, in particular, if $b$ is uniformly convex on $\dom(b)$, as follows from  \cite[Lemma 3.3]{ReemReichDePierro2019jour(BregmanEntropy)}.
\end{enumerate}
\end{remark}

\subsection{The Proximal Gradient Method}\label{subsec:PGM}
Here we briefly recall very well-known versions of the proximal gradient method (the proximal point algorithm). In its very basic form, this algorithm can be written as follows: 
\begin{equation}\label{eq:Prox}
x_{k}:=\argmin_{x\in C}(F(x)+c_k\|x-x_{k-1}\|^2), \quad\, k\geq 2,
\end{equation}
where  $F$ is the objective function to be minimized (it should satisfy certain assumptions, for example, to be proper, convex and lower semicontinuous), $C$ is the constraint subset over which the minimal value of $F$ is sought ($C$ is assumed to be a nonempty, closed and convex subset of the ambient space $X$),  $x_1\in X$ is some initial point, and $c_k$ is some positive parameter which may or may not depend on the iteration, for example, $c_k=0.5L$ for every $k\geq 2$, where $L$ is a fixed positive number. It is well known that $(x_k)_{k=1}^{\infty}$ is well defined (that is, the minimizer in \beqref{eq:Prox} exists and is unique)  and converges weakly, at least under certain assumptions. See, for example,   \cite[Theorem 3.4]{CombettesWajs2005jour} and \cite[p. 878 and Theorem 1 (p. 883)]{Rockafellar1976jour}. 

A further  generalization of \beqref{eq:Prox} is to replace the regularization term $\|x-x_k\|^2$ by a Bregman divergence $B(x,x_{k-1})$ and to replace the function $F$ by an approximation function $F_k$ so that the iterative scheme becomes 
\begin{equation}\label{eq:ProxBregman}
x_{k}:=\argmin_{x\in C}(F_k(x)+c_kB(x,x_{k-1})), \quad\, k\geq 2.
\end{equation}
One possible choice for $F_k$ is simply $F_k:=F$, as done in \cite{CensorZenios1992jour} (the paper which started the investigation of the proximal gradient method in the context of Bregman divergences). However, when $F=f+g$ and $f'$ exists, it is common to take $F_k(x):=f(x_{k-1})+\langle f'(x_{k-1}),x-x_{k-1}\rangle+g(x)$, $x\in C$, namely $F_k$ is the sum of $g$ with a linear term which approximates the smooth term $f$.  
A very partial list of references which consider, in various settings, the proximal gradient method with a Bregman divergence, is \cite{BeckTeboulle2003jour}, \cite{ButnariuIusemZalinescu2003jour}, \cite{ChenTeboulle1993jour}, \cite{OsherBurgerGoldfarbXuYin2005jour}, \cite{Tseng2008prep}, \cite{Tseng2010jour},  \cite{YinOsherGoldfarbDarbon2008jour}, \cite{Zaslavski2010jour}.

\section{TEPROG: a Telescopic Proximal Bregmanian Method}\label{sec:ISTA-Bregman}
In this section we present TEPROG, namely our new variants of the proximal gradient method with a Bregman divergence. We impose the following assumptions: The ambient space $(X,\|\cdot\|)$ is a real reflexive Banach space; $b:X\to]-\infty,\infty]$ is a semi-Bregman function (see Definition \bref{def:BregmanDiv}); in particular, its zone is $U:=\Int(\dom(b))\neq\emptyset$; we denote by $B$ the associated Bregman divergence of $b$ (defined in \beqref{eq:BregmanDistance}); we assume that $C\subseteq \dom(b)$, which is the constraint subset over  which we want to perform the minimization process, is a nonempty, closed and convex subset of $X$ which satisfies $C\cap U\neq\emptyset$; we are given a function $f:\dom(b)\to \R$ which is convex on $\dom(b)$ and G\^ateaux differentiable in $U$; we assume that the restriction of $f$ to $C$ is lower semicontinuous; we are also given a function $g:C\to ]-\infty,\infty]$ which is convex, proper and lower  semicontinuous.  Our goal is to solve the minimization problem 
\begin{equation}\label{eq:F-ISTA}
\inf\{F(x),\,\,x\in C\},
\end{equation}
where $F:X\to ]-\infty,\infty]$ is the function defined by 
\begin{equation}\label{eq:F=f+g}
F(x):=\left\{\begin{array}{lll}
f(x)+g(x),&\,x\in C,\\
\infty, & x\notin C.
\end{array}
\right.
\end{equation}
 We assume from now on that $C$ contains more than one point, since otherwise \beqref{eq:F-ISTA} is trivial. 
The assumptions on $f$, $g$ and $C$ imply that $F$ is convex, proper and lower semicontinuous. The values of $g$ outside $C$ and of $f$ outside $\dom(b)$ are not very important for us, but, as is common in optimization theory, we may assume that they are equal to $+\infty$. 

We define as follows a proximal function $p_{L,\mu,S}$ which depends on three parameters. 
The first two parameters are arbitrary $L>0$ and $\mu>0$. The third parameter is an arbitrary closed and convex subset $S\subseteq C$ which has the following properties: first, $b$ is strongly convex on $S$ with a strong convexity parameter $\mu>0$; second, $g$ is proper on $S$; third, $S\cap U\neq \emptyset$ (Assumption \bref{assum:MonotoneStronglyConvex} below ensures the existence of such a subset $S$). Given $y\in S\cap U$, let 
\begin{equation}\label{eq:p_L ISTA}
p_{L,\mu,S}(y):=\textnormal{argmin}\{Q_{L,\mu,S}(x,y): \quad x\in S\},
\end{equation}
where $Q_{L,\mu,S}(x,y)$ is defined as follows:
\begin{equation}\label{eq:Q_L-ISTA}
Q_{L,\mu,S}(x,y):=f(y)+\langle f'(y),x-y\rangle+\frac{L}{\mu}B(x,y)+g(x),\quad\forall x\in S, \forall  y\in  U.
\end{equation}
Our assumptions and Lemma \bref{lem:UniqueMinimzerQ} below ensure that the function $u_y(x):=Q_{L,\mu,S}(x,y)$, $x\in S$, has a unique minimizer in $S$. Thus $p_{L,\mu,S}(y)$ is well defined. We impose the following additional assumptions:
\begin{assumption}\label{assum:O(F)}
The set of minimizers of $F$, namely  
\begin{equation}
\textnormal{MIN}(F):=\Bigl\{x_{\textnormal{min}}\in C: F(x_{\textnormal{min}})=\inf\{F(x): x\in C\}\Bigr\}, 
\end{equation}
is  nonempty and contained in $U$.
\end{assumption}
\begin{assumption}\label{assum:MonotoneStronglyConvex}
There is a sequence of subsets $(S_k)_{k=1}^{\infty}$ in $C$, which, for the sake of convenience, we refer to as ``a telescopic sequence'', that has the following properties for each $k\in\N$: $S_k$ is closed  and convex; $S_k\cap U$ contains more than one point; $S_k\subseteq S_{k+1}\subseteq C$;  the function $b$ is strongly convex on each $S_k$ with a parameter $\mu_k$ such that  $\mu_{k+1}\leq \mu_k$; finally, $\cup_{k=1}^{\infty}S_k=C$. In addition, $g$ is proper on $S_1$ (and hence on each $S_k$, $k\in\N$). 
\end{assumption}
\begin{assumption}\label{assum:f'Lip}
 $f'$ is Lipschitz continuous on $S_k\cap U$ for each $k\in \N$. We denote the best (smallest) Lipschitz constant of $f'$ by $L(f',S_k\cap U):=\sup\{\|f'(x)-f'(y)\|/\|x-y\|: x\in S_k\cap U,\, y\in S_k\cap U,\, x\neq y\}$. 
\end{assumption}
\begin{assumption}\label{assum:pInU} 
For all $L>0$, $k\in \N$ and $y\in S_k\cap U$, we have $p_{L,\mu_k,S_k}(y)\in S_k\cap U$,  
where $\mu_k$ is the strong convexity parameter of $b$ over $S_k$. 
\end{assumption}
Assumptions \bref{assum:O(F)}--\bref{assum:pInU} occur frequently in applications. Indeed, the assumption that $\textnormal{MIN}(F)$ is nonempty (in Assumption \bref{assum:O(F)}) holds if, for example, $C$ is compact (in particular, when $X$ is finite-dimensional and $C$ is closed and bounded) or when both $g$  and the restriction of $f$ to $C$ are coercive,  since then $F$ is proper, coercive, convex, and lower semicontinuous on the closed and convex subset $C$ and hence, by a well-known and classical result \cite[Corollary 3.23, p. 71]{Brezis2011book}, it has a minimizer in $C$; the assumption that $\textnormal{MIN}(F)\subseteq U$ trivially holds if $C\subseteq U$, but it may hold in other cases as well (see, for instance, Examples \bref{ex:Simplex} and \bref{ex:Prism} below). Assumption \bref{assum:MonotoneStronglyConvex} holds frequently: for instance, if $b$ is strongly convex on bounded and convex subsets of $\dom(b)$, then we can take each $S_k$ to be the intersection of $C$ with a closed ball with center at some $y_0\in C\cap U$, where the radii of these balls are increasing, or, if $b$ is strongly convex on $C$, then we may take $S_k:=C$, $k\in\N$ (with the hope that the other assumptions will be satisfied with this choice).  Assumption \bref{assum:f'Lip} holds frequently, for instance if $f'$ is continuous on $U$ and $S_k$ is bounded for each $k$ and $f''$ exists and is bounded  on bounded subsets of $C\cap U$ (as follows from the generalized Mean Value Theorem \cite[Theorem 1.8, pp. 13, 23]{AmbrosettiProdi1993book}). Assumption \bref{assum:pInU} holds, for instance, when $C\subseteq U$ or (as follows from Lemma \bref{lem:zInSU} and Remark \bref{rem:SufficientConditions dom(g)dom(b)} below) when $X$ is finite-dimensional, $b$ is essentially smooth on $U$, and either $b$ is continuous on $\dom(b)$  or $g$ is continuous on $C$ or there exists a point $x\in S_1\cap U$ such that $g(x)\in\R$. Anyhow, Assumption \bref{assum:pInU} ensures  that $x_k$ defined below (in either \beqref{eq:x_k ISTA} or \beqref{eq:x_kBacktrack ISTA})  satisfies $x_k\in S_k\cap U$ for all $k\in \N$.

Our telescopic proximal gradient method is presented below.  We consider two versions of it:  one with a Lipschitz step size rule and the other with a backtracking step size rule. 
\begin{alg}\label{alg:LipschitzStep}$\,${\emph{\textbf{TEPROG with a Lipschitz step size rule:}}}\\
\noindent{\bf Input:} A positive number $L_1\geq L(f',S_1\cap U)$.\\ 
\noindent{\bf Step 1 (Initialization):} an arbitrary point $x_1\in S_1\cap U$.\\
\noindent{\bf Step $k$, $k\geq 2$: } $L_k$  is arbitrary such that $L_k\geq \max\{L_{k-1},L(f',S_k\cap U)\}$. Given $\mu_k>0$, which is a parameter of strong convexity of $b$ on $S_k$ (see Assumption \bref{assum:MonotoneStronglyConvex}), let  
\begin{equation}\label{eq:x_k ISTA}
x_{k}:=p_{L_k,\mu_k,S_k}(x_{k-1}).
\end{equation}
\end{alg}

\begin{alg}\label{alg:BacktrackingStep}$\,${\emph{\textbf{TEPROG with a backtracking  step size rule:}}}\\
\noindent{\bf Input: } $\eta>1$; if $S_k=C$ for all $k\in\N$, then another input is an arbitrary positive number $L_1$; otherwise, another input is any $L_1$ satisfying $0<L_1\leq \eta L(f',S_1\cap U)$.  \\
\noindent{\bf Step 1 (Initialization):} an arbitrary point $x_1\in S_1$.\\
\noindent{\bf Step $k$, $k\geq 2$: } Let $\mu_k>0$ be a parameter of strong convexity of $b$ on $S_k$ (the existence of $\mu_k$ is ensured by Assumption \bref{assum:MonotoneStronglyConvex}). Find the smallest nonnegative integer $i_k$ such that with $L_k:=\eta^{i_k}L_{k-1}$, we have 
\begin{equation}\label{eq:L_k_ISTA}
F(p_{L_k,\mu_k,S_k}(x_{k-1}))\leq Q_{L_k,\mu_k,S_k}(p_{L_k,\mu_k,S_k}(x_{k-1}),x_{k-1}).
\end{equation}
 Now let 
\begin{equation}\label{eq:x_kBacktrack ISTA}
x_k:=p_{L_k,\mu_k,S_k}(x_{k-1}). 
\end{equation}
\end{alg}

\begin{remark}\label{rem:L_kFiniteFx_k ISTA}
The backtracking step size rule is well defined in the sense that \beqref{eq:L_k_ISTA} does occur for some $i_k$. Indeed, given $2\leq k\in\N$, we first observe that \cite[Proposition 4.13(I)]{ReemReichDePierro2019jour(BregmanEntropy)}, when applied to $\psi(t):=0.5\mu_k t^2$, $t\in [0,\infty[$, $y:=x_{k-1}\in S_{k-1}\cap U\subseteq S_k\cap U$ and $x:=p_{L,\mu_k,S_k}(x_{k-1})\in S_k\cap U$, implies that $0.5\mu_k\|p_{L,\mu_k,S_k}(x_{k-1})-x_{k-1}\|^2\leq B(p_{L,\mu_k,S_k}(x_{k-1}),x_{k-1})$. By combining this inequality  with Lemma \bref{lem:LipschitzUpperBound} below (in which we take any $L\geq L(f',S_k\cap U)$) and with \beqref{eq:Q_L-ISTA}, we see that 
\begin{multline}\label{eq:FQ}
F(p_{L,\mu_k,S_k}(x_{k-1}))=f(p_{L,\mu_k,S_k}(x_{k-1}))+g(p_{L,\mu_k,S_k}(x_{k-1}))\leq \\
g(p_{L,\mu_k,S_k}(x_{k-1}))+f(x_{k-1})+\langle f'(x_{k-1}),p_{L,\mu_k,S_k}(x_{k-1})-x_{k-1}\rangle
+\frac{1}{2}L\|p_{L,\mu_k,S_k}(x_{k-1})-x_{k-1}\|^2\\
\leq g(p_{L,\mu_k,S_k}(x_{k-1}))+f(x_{k-1})+\langle f'(x_{k-1}),p_{L,\mu_k,S_k}(x_{k-1})-x_{k-1}\rangle
+\frac{L}{\mu_k} B(p_{L,\mu_k,S_k}(x_{k-1}),x_{k-1})\\
=Q_{L,\mu_k,S_k}(p_{L,\mu_k,S_k}(x_{k-1}),x_{k-1}).
\end{multline}
Since, by taking $i_k$ large enough, we can ensure that  $L_k\geq L(f',S_k\cap U)$, if we let $L:=L_k$, then \beqref{eq:FQ} implies that \beqref{eq:L_k_ISTA} holds. It may happen, however, that \beqref{eq:L_k_ISTA} holds even when $L_k<L(f',S_k\cap U)$. In the Lipschitz step size rule $L_k\geq L(f',S_k\cap U)$ for all $k\geq 2$ by the definition of $L_k$, and hence  \beqref{eq:L_k_ISTA} holds in this case too, as a result of \beqref{eq:FQ}. Finally, since  $Q_{L_k,\mu_k,S_k}(p_{L_k,\mu_k,S_k}(x_{k-1}),x_{k-1})\in \R$ according to Lemma  \bref{lem:UniqueMinimzerQ} below, we can conclude from \beqref{eq:FQ} that $F(x_{k})\in \R$ for all  $k\geq 2$. 
\end{remark}

\begin{remark}\label{rem:BoundLk ISTA}
 The constructions of $L_k$ in both the Lipschitz and the backtracking step size rules imply immediately that $(L_k)_{k=1}^{\infty}$ is increasing. Actually, by using simple induction it can be shown that in the backtracking step size rule, if  $S_j\neq C$ for some $j$, then $L_k\leq \eta L(f',S_k\cap U)$ for each $k\in\N$: see the appendix (Section \bref{sec:Appendix}) for the details. 
 
The above discussion implies that we can construct an increasing sequence $(\tau_k)_{k=1}^{\infty}$ of positive numbers which always satisfies 
\begin{equation}\label{eq:tau_rho ISTA}
L_k\leq \tau_k,\quad \forall\,k\in \N,
\end{equation}
and sometimes also satisfies 
\begin{equation}\label{eq:tau_eta}
\tau_k \leq \eta L(f',S_k\cap U),\quad \forall\,k\in \N.
\end{equation}
If we are interested only in \beqref{eq:tau_rho ISTA}, as in the Lipschitz step size rule, then, of course,  we can simply take $\tau_k:=L_k$ for all $k\in\N$, but we are free to select the $\tau_k$ parameters in any other way which guarantees \beqref{eq:tau_rho ISTA}. Suppose that we are in the backtracking step size rule and we are interested in both \beqref{eq:tau_rho ISTA} and \beqref{eq:tau_eta}. If $S_j\neq C$ for some $j\in\N$, then  $L_1\leq \eta L(f',S_1\cap U)$ and from the above discussion we can  take $\tau_k:=\eta L(f',S_k\cap U)$ for all $k\in \N$; this choice ensures that both \beqref{eq:tau_rho ISTA} and \beqref{eq:tau_eta} hold.  Otherwise $S_k=C$ for all $k\in \N$. In order to make sure that both \beqref{eq:tau_rho ISTA} and \beqref{eq:tau_eta} hold, the previous discussion implies that  we can take an arbitrary $L_1\leq \eta L(f',S_1\cap U)$ and define $\tau_k:=\eta L(f',S_k\cap U)$ for all $k\in \N$. If we are in the backtracking step size rule but select some $L_1>\eta L(f',S_1\cap U)$, then a simple analysis shows that actually $L_{k+1}=L_k$ for each $k\in\N$ and hence we can let $\tau_k:=L_k$ for each $k\in\N$: see Section \bref{sec:Appendix} for the simple details. 
\end{remark}

\section{The Convergence Theorem and Several Corollaries}\label{sec:Convergence}
In this section we present Theorem \bref{thm:BISTA-Convergence} below. It asserts that under some  assumptions, the iterative scheme produced by TEPROG (Algorithm \bref{alg:LipschitzStep} or Algorithm \bref{alg:BacktrackingStep}) converges at a certain non-asymptotic rate of convergence, and under further assumptions, also converges weakly. A few corollaries are presented after the theorem itself. 

\begin{thm}\label{thm:BISTA-Convergence}
In the framework of Section \bref{sec:ISTA-Bregman}, for each minimizer  $x_{\textnormal{min}}\in \textnormal{MIN}(F)$, there exists $k_0\in\N$ (which satisfies $k_0\geq 2$  if $F(x_1)=\infty$) such that for each  $k\geq k_0$, we have   
\begin{equation}\label{eq:O(1/k)}
F(x_{k+1})-F(x_{\textnormal{min}})\leq \frac{\tau_{k+1} B(x_{\textnormal{min}},x_{k_0})}{(k+1-k_0)\mu_{k+1}}.
\end{equation}
In addition, if 
\begin{equation}\label{eq:tau_mu_k}
\lim_{k\to\infty}\frac{\tau_{k}}{k\mu_k}=0, 
\end{equation}
then $(x_k)_{k=1}^{\infty}$ converges non-asymptotically to the minimal value of $F$. Moreover, if \beqref{eq:tau_mu_k} holds, $B$ has the limiting difference property and all of its first type level-sets are bounded, then there exists $z_{\infty}\in \textnormal{MIN}(F)$ such that $z_{\infty}=\lim_{k\to\infty}x_k$ weakly. In particular, if \beqref{eq:tau_mu_k} holds, $b'$ is weak-to-weak$^*$ sequentially continuous on $U$, and either $U$ is bounded or for each $x\in C$, there exists $r_x\geq 0$  such that $\{y\in U: \|y\|\geq r_x\}\neq\emptyset$ and $b$ is  uniformly convex relative to $(\{x\},\{y\in U: \|y\|\geq r_x\})$ with a gauge $\psi_x$ satisfying $\lim_{t\to\infty}\psi_x(t)=\infty$ (this latter condition holds, in particular, if $b$ is uniformly convex on $U$), then $(x_k)_{k=1}^{\infty}$ converges weakly to some point in $\textnormal{MIN}(F)$. 
\end{thm}

\begin{cor}\label{cor:AlmostO(1/k)}
Consider the framework of Section \bref{sec:ISTA-Bregman}, and  assume further that $f''$ exists, is bounded  and uniformly continuous on bounded subsets of $C\cap U$, and that $b$ is strongly convex on $C$ with a strong convexity parameter $\mu>0$. Then we can construct a sequence $(x_k)_{k=1}^{\infty}$, by either Algorithm \bref{alg:LipschitzStep} or Algorithm \bref{alg:BacktrackingStep}, which converges in the function values to a solution of \beqref{eq:F-ISTA}, at a rate of convergence which can be arbitrarily close to $O(1/k)$. In particular, for all $x_{\textnormal{min}}\in \textnormal{MIN}(F)$, $q\in ]0,1[$ , $y_0\in C\cap U$ and $\alpha>\|f''(y_0)\|$, there are a sequence $(x_k)_{k=1}^{\infty}$ and an index $k_0\in \N$ such that 
\begin{equation}\label{eq:k^q}
F(x_{k+1})-F(x_{\textnormal{min}})\leq \frac{k+1}{k+1-k_0}\cdot\frac{\alpha B(x_{\textnormal{min}},x_{k_0})}{\mu(k+1)^{1-q}},\quad \forall k\geq k_0.
\end{equation}
Moreover, if for each $x\in C$, there exists $r_x\geq 0$  such that $\{y\in U: \|y\|\geq r_x\}\neq\emptyset$ and $b$ is  uniformly convex relative to $(\{x\},\{y\in U: \|y\|\geq r_x\})$ with a gauge $\psi_x$ satisfying $\lim_{t\to\infty}\psi_x(t)=\infty$ (a condition which holds, in particular, when $b$ is uniformly convex on $U$), and if  $b'$ is weak-to-weak$^*$ sequentially continuous on $U$, then $(x_k)_{k=1}^{\infty}$  converges weakly to some point in $\textnormal{MIN}(F)$. 
\end{cor}

\begin{cor}\label{cor:O(1/k)}
In the framework of Section \bref{sec:ISTA-Bregman}, suppose that $f'$ is Lipschitz continuous on $C\cap U$ and  $b$ is strongly convex on $C$ with a strong convexity parameter $\mu>0$. Then by denoting  $S_k:=C$ and $\mu_k:=\mu$ for each $k\in \N$, the sequence $(x_k)_{k=1}^{\infty}$, which is obtained by either Algorithm \bref{alg:LipschitzStep} or Algorithm \bref{alg:BacktrackingStep}, converges in the function values to the minimal value of $F$ at a  rate of $O(1/k)$. Furthermore, $k_0$ (that is, the index which is guaranteed in the formulation of Theorem \bref{thm:BISTA-Convergence}) satisfies $k_0=1$, unless  $F(x_1)=\infty$ and then $k_0=2$. Moreover, if, in addition, $b'$ is weak-to-weak$^*$ sequentially continuous on $U$ and either $C$ is bounded or $b$ is uniformly convex on $\dom(b)$, then the above-mentioned sequence $(x_k)_{k=1}^{\infty}$ converges weakly to a solution of \beqref{eq:F-ISTA}. 
\end{cor}

\begin{remark}
It is possible to weaken a bit some of the assumptions needed for the definition of TEPROG and for the non-asymptotic convergence. Indeed, if $b$ is semi-Bregman with the exception of being strictly convex on $\dom(b)$, then $p_{L_k,\mu_k,S_k}(x_{k-1})$ will be a nonempty subset of $S_k$ (this is a consequence of Lemma \bref{lem:UniqueMinimzerQ}) and hence (by Assumption \bref{assum:pInU}) of $S_k\cap U$, and $x_k$ from either Algorithm \bref{alg:LipschitzStep} or Algorithm \bref{alg:BacktrackingStep} can be taken to be any point in $p_{L_k,\mu_k,S_k}(x_{k-1})$. The proof of the non-asymptotic convergence in Theorem \bref{thm:BISTA-Convergence} remains as it is. However, it is an open problem whether the weak convergence result holds when $b$ is no longer strictly convex, since the derivation of this convergence result is based on Lemma \bref{lem:BregmanLimit}, which is based on the assumption that $b$ is strictly convex. 
\end{remark}

\section{Examples}\label{sec:Examples}
In this section we present a few examples which illustrate our convergence results. 
\begin{expl}\label{ex:ell_p-ell_1}{\bf ($\ell_p-\ell_1$ optimization): }  
We use the notation of Section \bref{sec:ModelProblem} and show how to solve, using the Lipschitz step size rule version of TEPROG, the $\ell_p-\ell_1$ optimization problem \beqref{eq:ell_p-ell_1} mentioned there. For doing this, we need to choose the sets $S_k$ and the Bregman function $b$, to estimate $\mu_k$, $L_k$ and $\tau_k$, and also to show how to compute $x_k$ (from \beqref{eq:x_k ISTA}). 

Let $X:=\R^n$ and $Y:=\R^m$. Denote $C:=X$. Fix some $r\in [2,\infty[$ and endow $X$ with the norm $\|x\|_r=(\sum_{j=1}^n |x_j|^r)^{1/r}$. For each $k\in\N$, let $S_k:=[-\rho_k,\rho_k]^n$, where $(\rho_k)_{k=1}^{\infty}$ is an increasing sequence of positive numbers which satisfies both $\lim_{k\to\infty}\rho_k=\infty$ and $\lim_{k\to\infty}\rho_k^{p-2}/k=0$ (for instance, we can take $\rho_k:=k^\sigma$, where $\sigma$ is a fixed number which  satisfies $\sigma\in ]0,1/(p-2)[$ if $p>2$ and can be an arbitrary positive number if $p=2$). Then $S_k\subseteq S_{k+1}$ for all $k\in\N$ and $\cup_{k=1}^{\infty}S_k=C=\R^n$. 

Denote $b(x):=\frac{1}{2}\|x\|_2^2$ for each $x\in X$. It is well known and can easily be proved that $b$ is strongly convex on $\R^n$ with the  Euclidean norm, where the strong convexity parameter is 1: see, for instance, \cite[Subsection 11.4]{ReemReichDePierro2019jour(BregmanEntropy)}. Since $r\geq 2$, it follows that $\|x\|_2\geq \|x\|_r$ for every $x\in \R^n$. This inequality and simple calculations show  that $b$ is strongly convex on $(X,\|\cdot\|_r)$, again with 1 as a strong convexity parameter (for a more general statement, see \cite[Proposition 5.3]{ReemReichDePierro2019jour(BregmanEntropy)}). Thus $b$ is strongly convex on $S_k$ (for each $k\in \N$) with $\mu_k:=1$ as a strong convexity parameter. 

Let $h(y):=\frac{1}{p}\|y\|_p^p$ for all $y\in Y$. Given a nonempty and bounded subset $T$ of $Y$, since $p\in [2,\infty[$, it  essentially follows from \cite[pp. 48--49]{ButnariuIusemZalinescu2003jour} that $h'$ is Lipschitz continuous over $T$ with a Lipschitz constant $(p-1)(2M_T)^{p-2}$, where $M_T$ is an upper bound on the norms of the elements of $T$ (but we note that there is a small mistake in \cite[Expression (26)]{ButnariuIusemZalinescu2003jour}: instead of the inequality $\|h'(x)-h'(y)\|\leq (p-1)(\|x\|_p+\|y\|_p)^{(p-2)/p}\|x-y\|_p$ written there, the following expression should be written: $\|h'(x)-h'(y)\|\leq (p-1)(\|x\|_p+\|y\|_p)^{p-2}\|x-y\|_p$).  Since $f(x)=h(Ax-c)$ for every $x\in X$, a direct calculation shows that $f'$ is Lipschitz continuous on $S$ with $(p-1)(2M_{AS-c})^{p-2}\|A\|$ as a Lipschitz constant, where $M_{AS-c}$ is an upper bound on the norms of the elements of the set $AS-c$ and $\|A\|:=\sup\{\|Ax\|_p: \|x\|_r=1\}$ is the operator norm of $A$. Since the norm of any element in $AS-c$ is bounded by $\|A\|M_S+\|c\|_p$, where $M_S$ is an upper bound on the norms of the elements of $S$, and since $\|x\|_r\leq n^{1/r}\rho_k$  for all $x\in S_k$, we conclude that for all $k\in\N$, the function $f'$ is Lipschitz continuous on $S_k$ with $L_k:=(p-1)2^{p-2}\|A\|(\|A\|n^{1/r}\rho_k+\|c\|_p)^{p-2}$ as a Lipschitz constant. Then $(L_k)_{k=1}^{\infty}$ is an increasing sequence and Remark \bref{rem:BoundLk ISTA} ensures that we can take $\tau_k:=L_k$ for all $k\in\N$.

Fix $2\leq k\in\N$. We need to compute $x_k$. It follows from \beqref{eq:x_k ISTA}, \beqref{eq:p_L ISTA} and \beqref{eq:Q_L-ISTA} that 
\begin{equation}\label{eq:x_k ell_p-ell_1}
x_k=\argmin_{w\in S_k}\left\{f(x_{k-1})+\langle f'(x_{k-1}),w-x_{k-1}\rangle+\frac{L_k}{2\mu_k}\|w-x_{k-1}\|_2^2+\lambda\|w\|_1\right\}.
\end{equation}
Denote by $\phi_j$ the $j$-th component of $f'(x_{k-1})$, $j\in \{1,\ldots,n\}$. Direct calculations show that $\phi_j=\sum_{i=1}^m|(Ax_{k-1})_i-c_i|^{p-2}((Ax_{k-1})_i-c_i)A_{ij}$, where $(Ax_{k-1})_i$ is the $i$-th component of the vector $Ax_{k-1}\in \R^m$ and $A_{ij}$ is the $(i,j)$-entry of the matrix representation of $A$ for each $i\in\{1,\ldots,m\}$ and $j\in\{1,\ldots,n\}$. Denote $\alpha_{k-1}:=f(x_{k-1})-\langle f'(x_{k-1}), x_{k-1}\rangle$. In addition, for each $j\in \{1,\ldots,n\}$ let $H_j:[-\rho_k,\rho_k]\to\R$ be the function defined by $H_j(t):=\phi_j t+0.5(L_k/\mu_k)(t-x_{k-1,j})^2+\lambda|t|$, $t\in [-\rho_k,\rho_k]$, where $x_{k-1,j}$ is the $j$-th component of $x_{k-1}$. These notations and \beqref{eq:x_k ell_p-ell_1} imply that 
\begin{equation}\label{eq:x_k ell_p-ell_1 V2}
x_k=\argmin\left\{\alpha_{k-1}+\sum_{j=1}^n H_j(w_j):\quad (w_j)_{j=1}^n\in [-\rho_k,\rho_k]^n\right\}.
\end{equation}
Thus, in order to compute $x_k$ it is sufficient to find, for each $j\in\{1,\ldots,n\}$, a minimizer of $H_j$ on $[-\rho_k,\rho_k]$. Since $H_j$ is differentiable on the set $[-\rho_k,\rho_k]\backslash\{0,-\rho_k,\rho_k\}$ and $H_j'(t)=\phi_j+(L_k/\mu_k)(t-x_{k-1,j})+\lambda\cdot\sign(t)$ for each $t\in [-\rho_k,\rho_k]\backslash\{0,-\rho_k,\rho_k\}$, considerations from elementary calculus show that the minimal value of $H_j$ is attained, and it can be attained only at one of the following (at most) five points: $t_j(1):=-\rho_k$, $t_j(2):=\rho_k$, $t_j(3):=0$,  $t_j(4):=(\mu_k/L_k)(-\phi_j-\lambda)+x_{k-1,j}$ (only if the expression which defines $t_j(4)$  is in $]0,\rho_k[$), and $t_j(5):=(\mu_k/L_k)(-\phi_j+\lambda)+x_{k-1,j}$ (only if the expression which defines $t_j(5)$ is in $]-\rho_k,0[$). Now we merely need to compute $H_j(t_j(1)),\ldots, H_j(t_j(5))$ and to find the minimal value among them. Then we let $\wt{w}_j$ be the corresponding argument among the $t_j(1),\ldots,t_j(5)$  which leads to this minimal value ($\wt{w}_j$ is unique since $H_j$ is strictly convex). By repeating this process for all $j\in\{1,\ldots,n\}$ and using \beqref{eq:x_k ell_p-ell_1 V2} we see that $x_k=(\wt{w}_j)_{j=1}^n$. 

Finally, since $\lim_{\|x\|\to\infty}F(x)=\infty$ and $F$ is continuous over $X$, a well-known result in classical analysis implies that $\textnormal{MIN}(F)\neq\emptyset$. We conclude that Assumptions \bref{assum:O(F)}--\bref{assum:pInU} hold. Since $\lim_{k\to\infty}\rho_k^{p-2}/k=0$, Theorem \bref{thm:BISTA-Convergence} implies that the proximal sequence $(x_k)_{k=1}^{\infty}$ which is obtained from Algorithm \bref{alg:LipschitzStep} converges to a point in $\textnormal{MIN}(F)$, and \beqref{eq:O(1/k)} implies that the non-asymptotic rate of  convergence is $O(\rho_k^{p-2}/k)$. In particular, if $\rho_k:=k^{\sigma}$ for all $k\in\N$, where $\sigma\in ]0,1/(p-2)[$ if $p>2$ and can be an arbitrary positive number if $p=2$, then the non-asymptotic rate of  convergence is $O(1/k^{1-\sigma(p-2)})$, that is, by letting $\sigma$ be arbitrarily close to 0, the non-asymptotic rate of convergence can be arbitrarily close to $O(1/k)$.
\end{expl}

\begin{expl}\label{ex:Simplex}
Fix some $p\in [1,\infty]$ and let $(X,\|\cdot\|)$ be $\R^3$ with the $\ell_p$ norm. Let $C:=C_0$ where  $C_0:=\{w=(w_1,w_2,w_3)\in[0,1]^3: \sum_{j=1}^3 w_j=1\}$ is the probability simplex. Let $U:=]0,\infty[^3$ and $b$ be the negative Boltzmann-Gibbs-Shannon entropy function defined by $b(w):=\sum_{j=1}^3 w_j\log(w_j)$ for $w=(w_j)_{j=1}^3\in U$ and $b(w):=0$ for $w$ on the boundary of $U$ and $b(w):=\infty$ for $w\notin \cl{U}$. Let $f(w):=\frac{4}{15}\left((w_1+w_2)^{5/2}+(w_2+w_3)^{5/2}+(w_3+w_1)^{5/2}\right)$ for all $w\in \cl{U}=\dom(b)$. Let $S_k:=C$ for all $k\in\N$. Let $I$ be a nonempty finite set and for each $i\in I$, let $g_i:X\to\R$ be the linear  function defined for each  $w\in X$ by $g_i(w):=\sum_{j=1}^3 a_{ij}w_j$, where $a_{ij}\in \R$ for each $i\in I$ and $j\in \{1,2,3\}$, and $\sum_{j=1}^3 a_{ij}\leq 1$  for each $i\in I$. Assume also that $\max\{\min\{a_{i1},a_{i2},a_{i3}\}: i\in I\}\geq 0.27$. This condition holds, for instance, when $\min\{a_{i1},a_{i2},a_{i3}\}\geq 0.27$ for all $i\in I$. Let $g(w):=\max\{g_i(w): i\in I\}$ for all $w\in C$. Then $g$ is convex and continuous, but usually non-smooth.  

Since $C$ is bounded, it follows that $b$ is strongly convex on $C$ (see \cite[Subsection 6.3]{ReemReichDePierro2019jour(BregmanEntropy)}). A somewhat technical but simple verification (the argument is based on the component-wise monotonicity of our norm and its dual, the fact that $|w_i|\leq \|(w_1,w_2,w_3)\|$ for all $i\in\{1,2,3\}$, and also on the mean value theorem applied to the function $t\mapsto (2/3)t^{3/2}$ on a bounded interval contained in $[0,\infty[$) shows that $f'$ is Lipschitz continuous  on any bounded and  convex subset of $U\cap C$ with $4\sqrt{2}\|(1,1,1)\|_q\sqrt{M_C}$ as a (not necessarily optimal) global  Lipschitz constant; here $\|\cdot\|_q$ is the dual norm (namely $(1/p)+(1/q)=1$) and $M_C$ is the radius of a ball  which contains $C$  and has 0 as its center (any such ball is fine). In addition, $F:=f+g$ is convex and continuous on the compact subset $C$ and hence $\textnormal{MIN}(F)\neq\emptyset$. Moreover, by considerations from elementary calculus we have $F(w)\geq(4/15)(1+2^{-3/2})+0.27>0.63$ for each $w$ in the intersection of $C$ with the boundary of $U$ (that is, any $w$ which belongs to the union of the segments $\{(y_1,y_2,0)\in [0,1]^3: y_1+y_2=1\}$, $\{(y_1,0,y_3)\in [0,1]^3: y_1+y_3=1\}$, $\{(0,y_2,y_3)\in [0,1]^3: y_2+y_3=1\}$) and also that $F(c)<0.63$ for $c=(1/3,1/3,1/3)\in C$. Since obviously $F(x)\leq F(c)$ for every $x\in \textnormal{MIN}(F)$, it follows that no point in $\textnormal{MIN}(F)$ can belong to the intersection of $C$ with the boundary of $U$. Hence  $\textnormal{MIN}(F)\subset U$. Since $X$ is finite-dimensional, $b'$ is  weak-to-weak$^*$. Finally, direct calculation shows that  $b$ is essentially smooth and hence Assumption \bref{assum:pInU} holds (see the discussion after Assumption \bref{assum:pInU}). Thus all the conditions needed in Corollary \bref{cor:O(1/k)} are satisfied. Therefore the  proximal sequence $(x_k)_{k=1}^{\infty}$ which is obtained from either Algorithm \bref{alg:LipschitzStep} or  \bref{alg:BacktrackingStep} converges to a point in $\textnormal{MIN}(F)$, and the function values  rate of  convergence is $O(1/k)$. 
\end{expl}
\begin{expl}\label{ex:Prism}
Consider the setting of Example \bref{ex:Simplex}, where we re-define $C$ to be the prism which is obtained from the intersection of the halfspaces $\{w\in \R^3: w_1+w_2+w_3\geq 1\}$ (on the boundary of which the probability simplex $C_0$ is located), $\{w\in \R^3: -2w_1+w_2+w_3\leq 1\}$, $\{w\in \R^3: w_1-2w_2+w_3\leq 1\}$, and $\{w\in \R^3: w_1+w_2-2w_3\leq 1\}$. Now $C$ is unbounded, but  since $F$ is continuous on $C$ and $\lim_{\|w\|\to \infty, w\in C}F(w)=\infty$, it follows that  $\textnormal{MIN}(F)\neq \emptyset$. Since the intersection of $C$ with the boundary of $U$ is as in Example \bref{ex:Simplex} (namely, the boundary of $C_0$), we have $\textnormal{MIN}(F)\subset U$, that is, Assumption \bref{assum:O(F)} holds. For each $k\in\N$, let $V_k$ be the closed ball of radius $r_k:=\sqrt{k}$  and center at the origin and let $S_k:=C\cap V_k$.  It follows from \cite[Subsection 6.3]{ReemReichDePierro2019jour(BregmanEntropy)}  that $b$ is strongly  convex on $S_k$ with  $\mu_k:=\beta/r_k$ as a strong convexity parameter, where $\beta$ is some  positive constant not depending on $k$ (note that one should not expect $b$ to be globally strongly convex, since \cite[Subsection 6.5]{ReemReichDePierro2019jour(BregmanEntropy)} shows that $b$ is not  even uniformly convex on $U$). Thus Assumption \bref{assum:MonotoneStronglyConvex} holds. As we saw in Example \bref{ex:Simplex}, $b$ is essentially smooth and hence Assumption \bref{assum:pInU} holds (see the discussion after Assumption \bref{assum:pInU}). The same reasoning  as the one used in Example \bref{ex:Simplex} can be used to show that $f'$ is Lipschitz continuous  on $S_k\cap U$ with a Lipschitz constant $L_k:=4\sqrt{2}\|(1,1,1)\|_q\sqrt{r_k}=O(k^{0.25})$. Hence $(L_k)_{k=1}^{\infty}$ is increasing and Assumption \bref{assum:f'Lip} holds. Finally, \cite[Subsection 6.4]{ReemReichDePierro2019jour(BregmanEntropy)} shows that for each $x\in \cl{U}$ (in particular, for each $x\in C$), there exists $r_x\geq 0$ such that $b$ is uniformly convex relative to $(\{x\},\{y\in U: \|x\|\geq r_x\})$, with some gauge $\psi_x$ satisfying $\lim_{t\to\infty}\psi_x(t)=\infty$ (namely, $r_x:=2\|x\|$ and $\psi_x(t)=\gamma t$, $t\in [0,\infty[$, for some $\gamma>0$ independent of $t$). Thus, if we denote $\tau_k:=L_k$ for every $k\in\N$, then Theorem \bref{thm:BISTA-Convergence} and Remark \bref{rem:BoundLk ISTA} imply that the proximal sequence $(x_k)_{k=1}^{\infty}$ generated by Algorithm \bref{alg:LipschitzStep}  converges to a point in $\textnormal{MIN}(F)$, and \beqref{eq:O(1/k)} implies that the non-asymptotic rate of  convergence is $O(1/k^{0.25})$.
\end{expl}
\begin{expl}
Let $X:=\ell_2$ with the norm  $\|(x_i)_{i=1}^{\infty}\|:=\sum_{i=1}^{2n}|x_i|+\sqrt{\sum_{i=2n+1}^{\infty}x_i^2}$ for all $x=(x_i)_{i=1}^{\infty}\in X$, where $n\in\N\cup\{0\}$ and $\sum_{i=1}^{2n}|x_i|:=0$ if $n=0$. A simple verification shows that $(X,\|\cdot\|)$ is isomporphic to $(X,\|\cdot\|_{\ell_2})$. Let $C:=\{x\in X:\, x_i\geq 0\,\,\forall i\in\N\}$ be the nonnegative orthant and consider the function 
\begin{equation*}
b(x):=\left\{
\begin{array}{lll}
\sum_{i=1}^{\infty}\left(e^{(x_{2i-1}+x_{2i})^2}+e^{(x_{2i-1}-x_{2i})^2}-2\right), & x\in C,\\
\infty, & \textnormal{otherwise}.
\end{array}\right.
\end{equation*}
Considerations similar to the ones presented in \cite[Section 10]{ReemReichDePierro2019jour(BregmanEntropy)} show that this function is a well-defined semi-Bregman function which satisfies various additional properties, including the limiting difference property. Moreover, these considerations show that if $n=0$, then $b$ is strongly convex on $C$ with $\mu=4$ as a strong convexity parameter, and if $n>0$, then $b$ is strongly convex on $C$ with $\mu=1/n$ as a strong convexity parameter. Fix some $\beta\geq 2$ and a sequence $(p_i)_{i=1}^{\infty}$ of real numbers in $[2,\beta]$, and let 
$f(x):=\sum_{i=1}^{\infty}x_i^{p_i}$ for every $x\in C$. Then $f$ is well defined, convex and continuous on $C$ and differentiable in $U:=\Int(\dom(b))=\{x\in X:\, x_i>0\,\,\forall i\in\N\}$.  Fix some $\lambda>0$ (a regularization parameter) and define $g(x):=\lambda\|x\|$ for each $x\in C$. Let $F:=f+g$.  For each $2\leq k\in\N$, let $S_k$ be the intersection of $C$ with the ball of radius $r_k:=k^{\sigma}$ and center 0, where $\sigma\in ]0,1/(\beta-2)[$ is fixed in advance (if $\beta=2$, then $\sigma$ can be an arbitrary positive number). In addition, let $S_1:=S_2$ and $r_1:=r_2$. Since $b$ is strongly convex on $C$ with a strong convexity parameter $\mu>0$, it is strongly convex on $S_k$ for each $k\in\N$ with $\mu_k:=\mu$ as its  strong convexity parameter. 

We claim that $f'$, which exists in $U$, is Lipschitz continuous on $S_k\cap U$ for each $k\in\N$. Indeed, we first observe that if, given $i\in\N$, we define $h_i(t):=t^{p_i-1}$ for every $t\in [0,r_k]$, then the mean value theorem implies that for each $t,s\in [0,r_k]$, there exists some $\theta$ between $t$ and $s$ such that $h_i(t)-h_i(s)=h'(\theta)(t-s)$. Hence $|h_i(t)-h_i(s)|=(p_i-1)\theta^{p_i-2}|t-s|\leq (\beta-1)r_k^{p_i-2}|t-s|\leq (\beta-1)r_k^{\beta-2}|t-s|$, where we used in the last inequality the assumptions that $\beta\geq 2$ and $r_k\geq 1$ for all $k\in \N$. Since $f'(x)(w)=\sum_{i=1}^{\infty}p_i x_i^{p_i-1}w_i$ for every $x\in U$ and $w\in X$, the above inequality, the Cauchy-Schwarz inequality and simple  calculations imply that for all $x,y\in S_k\cap U$,
\begin{multline*}
\|f'(x)-f'(y)\|=\sup_{w\in X, \|w\|=1}|f'(x)(w)-f'(y)(w)|\\
=\sup_{w\in X, \|w\|=1}\left|\sum_{i=1}^{\infty}p_iw_i(x_i^{p_i-1}-y_i^{p_i-1})\right|
\leq \beta
\sup_{w\in X, \|w\|=1} \sqrt{\sum_{i=1}^{\infty}w_i^2}\sqrt{\sum_{i=1}^{\infty}|x_i^{p_i-1}-y_i^{p_i-1}|^2}\\
=\beta\|w\|\sqrt{\sum_{i=1}^{\infty}|h(x_i)-h(y_i)|^2}\leq \beta\sqrt{\sum_{i=1}^{\infty}(\beta-1)^2r_k^{2(\beta-2)}|x_i-y_i|^2}\\
=\beta(\beta-1)r_k^{\beta-2}\|x-y\|_{\ell_2}\leq \beta(\beta-1)r_k^{\beta-2}\|x-y\|,
\end{multline*}
namely $f'$ is Lipschitz continuous on $S_k\cap U$ with $L_k:=\beta(\beta-1)r_k^{\beta-2}$ as a Lipschitz constant. Since $\lim_{\|x\|\to\infty, x\in C}F(x)=\infty$, and $F$ is convex and continuous over the nonempty, closed and convex subset $C$, a well-known result \cite[Corollary 3.23, p. 71]{Brezis2011book} implies that $\textnormal{MIN}(F)\neq\emptyset$. We conclude that Assumptions \bref{assum:O(F)}--\bref{assum:pInU} hold. Hence, if we let $\tau_k:=L_k$ for all $k\in\N$, then Theorem \bref{thm:BISTA-Convergence} implies that the proximal sequence $(x_k)_{k=1}^{\infty}$ generated by Algorithm \bref{alg:LipschitzStep} converges weakly to a point in $\textnormal{MIN}(F)$, and \beqref{eq:O(1/k)} implies that the non-asymptotic rate of  convergence is $O(1/k^{1-\sigma(\beta-2)})$.
\end{expl}

\begin{expl}\label{ex:NonHilbertian}
Let $p\in ]1,2]$ be given. Let $X$ be $\ell_p$ with the $\|\cdot\|_p$ norm. This is a reflexive Banach space which is not isomorphic to a Hilbert space unless $p=2$. Let $C:=X$ and $b:X\to\R$ be defined by $b(y):=\frac{1}{2}\|y\|_p^2$, $y\in X$. Then $U:=\Int(\dom(b))=X$.  Define $h:X\to\R$ by $h(y):=\sum_{i=1}^{\infty}|y_i|^{p+2}$. Suppose that $0\neq A:X\to X$ is a given bounded linear operator and that $c\in X$ is a given vector. Let $f:X\to\R$ and $g:X\to\R$ be defined by $f(y):=h(Ay-c)$ and $g(y):=\max\{\lambda\|y\|_p,\sup\{\lambda_i |y_i|: i\in\N\}\}$ for all $y=(y_i)_{i=1}^{\infty}\in X$, respectively, where $\lambda>0$ is given and $(\lambda_i)_{i=1}^{\infty}$ is a given bounded sequence of positive parameters. The facts that $p\in ]1,2]$ and $\lim_{i\to\infty}y_i=0$ for each $y\in X$ imply that $h$ and hence $f$ are  well defined and smooth. In addition, $f$ is convex. Let $F:=f+g$. Then $g$ and $F$ are convex, proper and lower semicontinuous. 

We want to solve \beqref{eq:F-ISTA}. Since $F$ is also coercive, it has a minimizer and hence Assumption \bref{assum:O(F)} holds. Direct calculations show that $b':X\to X^*\cong \ell_q$ (where $q=p/(p-1)$) exists and satisfies $b'(y)=\|y\|_p^{2-p}(|y_i|^{p-1}\sign(y_i))_{i=1}^{\infty}$ for all $y\in X$. In addition, $b$ is strongly convex on $X$ as follows, for instance, from \cite[Example 6.7]{ReemReich2018jour} (since there $\rho=\max\{2,p\}=2$, and \cite[Inequality (6.8)]{ReemReich2018jour} is equivalent to strong convexity as follows from \cite[Corollary 3.5.11(i)--(v), pp. 217--218]{Zalinescu2002book}; in this connection, we note that the fact that $b$ is strongly convex has been observed long before, for example in \cite[p. 314]{Reich1986incol}). Denote by $\mu$ the strong convexity parameter of $b$ and let $\mu_k:=\mu$ for all $k\in\N$. It follows from Remark \bref{rem:BregmanDivergence}\beqref{item:LevelSetBounded} that all the first type level-sets of $B$ are bounded. 

For each $k\in\N$, let $S_k$ be the closed ball of radius $\rho_k$ about the origin, where $\lim_{k\to\infty}\rho_k^p/k=0$, $\lim_{k\to\infty}\rho_k=\infty$ and $\rho_{k+1}\geq \rho_k$ for all $k\in\N$ (for example, we can fix $\sigma\in ]0,1/p[$ and take $\rho_k:=k^{\sigma}$). The previous lines immediately imply that Assumption \bref{assum:MonotoneStronglyConvex} and \bref{assum:pInU} hold. To see that also Assumption \bref{assum:f'Lip} holds, we need to show that $f'$ is Lipschitz continuous. In order to show  this, we first observe that since $f'(y)=h'(Ay-c)A$ for every $y\in X$, we have 
\begin{multline*}
\|f'(u)-f'(v)\|_q=\sup_{\|w\|_p=1}|(h'(Au-c)-h'(Av-c))(Aw)|\leq \|h'(Au-c)-h'(Av-c)\|_q\|A\|
\end{multline*} 
for all $u$ and $v$ in $X$. Since  $\|Ay-c\|_p\leq \|A\|\rho_k+\|c\|_p$ for all $y\in S_k$, it is sufficient to show that $h'$ is Lipschitz continuous on the ball of radius $\|A\|\rho_k+\|c\|_p$ about the origin. 

Indeed, direct calculations show that $h'(y)=(\phi(y_i))_{i=1}^{\infty}$, $y\in X$, where $\phi:\R\to\R$ is defined by $\phi(s):=(p+2)|s|^{p+1}\sign(s)$, $s\in \R$. Elementary considerations (based on the Taylor expansion with remainder in Lagrange's form and the fact that $\phi''$  exists and is continuous and bounded on any compact interval) imply that $h''$ exists. Moreover, these considerations imply that $h''$  satisfies the equality $h''(y)(w,\wt{w})=(p+2)(p+1)\sum_{i=1}^{\infty}|y_i|^p w_i\wt{w}_i$ for all $y\in X$ and $w$, $\wt{w}\in X$ which satisfy $\|w\|_p=1=\|\tilde{w}\|_p$. In particular, $h'$ is continuous. Since $|w_i|\leq \|w\|_p=1$ for all $i\in\N$, one has $|h''(y)(w,\wt{w})|\leq (p+2)(p+1)\sum_{i=1}^{\infty}|y_i|^p=(p+2)(p+1)\|y\|_p^p$. Since $\|h''(y)\|=\sup_{\|w\|_p=1, \|\wt{w}\|_p=1}|h''(y)(w,\wt{w})|$, the previous lines imply that $\|h''\|$ is bounded by $(p+2)(p+1)M^p$ on the ball of radius $M>0$ about the origin. Therefore the (generalized) Mean Value Theorem applied to $h'$ (see \cite[Theorem 1.8, p. 13, and also p. 23]{AmbrosettiProdi1993book}) implies that $h'$ is Lipschitz continuous on this ball with $(p+2)(p+1)M^p$ as a Lipschitz constant. This is true, in particular, for $M:=\|A\|\rho_k+\|c\|_p$, and so the previous lines imply that $f'$ is Lipschitz continuous on $S_k$ with a Lipschitz constant $L_k:=(p+2)(p+1)(\|A\|\rho_k+\|c\|_p)^p\|A\|^2$.

Since $L_k=O(\rho_k^p)$, by letting $\tau_k:=L_k$ and using our assumption that $\lim_{k\to\infty}\rho_k^p/k=0$, we have $\lim_{k\to\infty}\tau_k/(k\mu_k)=0$. As a result, if $(x_k)_{k=1}^{\infty}$ is the sequence defined by Algorithm \bref{alg:LipschitzStep}, then Theorem \bref{thm:BISTA-Convergence}  implies that it  converges non-asymptotically to the minimal value of $F$ and the rate of non-asymptotic convergence is $O(\rho_k^p/k)$. We note that it is not clear whether $(x_k)_{k=1}^{\infty}$ converges weakly to a minimizer of $F$, since $b'$ is not weak-to-weak$^*$ sequentially continuous unless $p=2$ (for instance, if $(e_k)_{k\in\N}$ is the canonical basis in $X$, then $(e_k+e_1)_{k=1}^{\infty}$ converges weakly to $e_1$ but $\lim_{k\to\infty}\langle b'(e_k+e_1),e_1\rangle=2^{(2-p)/p}\neq 1=\langle b'(e_1),e_1\rangle$). The issue of weak convergence of $(x_k)_{k=1}^{\infty}$ is left as an open problem for the future.
\end{expl}

\section{Proofs}\label{sec:Proofs}
This section contains the proofs of Theorem \bref{thm:BISTA-Convergence} and the corollaries which follow it. The proofs are based on several auxiliary  assertions. 

We start with the following general lemma. This lemma and its proof are minor modifications of \cite[Lemma 3.4]{Reem2012incol}. For the sake of completeness, its proof is given in the appendix (Section \bref{sec:Appendix}). 

\begin{lem}\label{lem:BregmanLimit}
 Suppose that  $X\neq \{0\}$ is a real normed space and $b:X\to (-\infty,\infty]$ is a semi-Bregman function with zone $U$. Let $B:X^2\to (-\infty,\infty]$ be the Bregman divergence associated with $b$ and defined in \beqref{eq:BregmanDistance}, and suppose further that $B$  satisfies the limiting difference property. If $(x_k)_{k=1}^{\infty}$ is a bounded sequence in $U$ having the properties that all of its weak  cluster points  are in $U$ and $\lim_{k\to\infty}B(q,x_k)$ exists and is finite for each weak cluster point $q\in X$ of $(x_k)_{k=1}^{\infty}$, then  $(x_k)_{k=1}^{\infty}$  converges weakly to a point in $U$.  
\end{lem}

The next lemma, which is probably known, generalizes \cite[Proposition 11.14, p. 193]{BauschkeCombettes2017book} from real Hilbert spaces to real normed spaces. The proof appears in the appendix (Section \bref{sec:Appendix}). 
\begin{lem}\label{lem:SuperCoercive}
Suppose that $X\neq \{0\}$ is a real normed space and let $\emptyset\neq S\subseteq X$ be closed, convex and unbounded. Let $u:S\to(-\infty,\infty]$ be defined by $u(x):=v(x)+w(x)$ for each $x\in S$, 
where $v:S\to(-\infty,\infty]$ is convex, proper, and lower semicontinuous, and $w:S\to (-\infty,\infty]$ is supercoercive on $S$. Then $u$ is supercoercive on $S$.
\end{lem}
The next lemma is needed for the formulation of the proximal forward-backward algorithms presented in Section \bref{sec:ISTA-Bregman}. 
\begin{lem}\label{lem:UniqueMinimzerQ} 
Let $X\neq \{0\}$ be a real reflexive Banach space. Suppose that $b:X\to(-\infty,\infty]$ is lower semicontinuous, convex and proper on $X$, that $U:=\Int(\dom(b))\neq\emptyset$, and that $b$ is G\^ateaux differentiable in $U$. Suppose that $S\subseteq \dom(b)$ is  nonempty, closed and convex,  and that $b$ is strictly convex on $S$. Suppose also that $f:\dom(b)\to\R$ is G\^ateaux differentiable in $U$ and that $g:S\to(-\infty,\infty]$ is proper, lower semicontinuous, and convex.  Fix arbitrary $L>0$ and $\mu>0$, let $B$ be the  associated Bregman divergence of $b$, and let $Q_{L,\mu,S}$ be defined in \beqref{eq:Q_L-ISTA}. 
Fix some $y\in U$ and assume that at least one of the following conditions holds:
\begin{enumerate}[(i)]
\item\label{item:SisBounded} $S$ is bounded;
\item\label{item:SisUnboundedRelativeUC} $S$ is unbounded and $b$ is uniformly convex relative to $(S,\{y\})$ with a relative gauge $\psi$ which satisfies $\lim_{t\to\infty}\psi(t)/t=\infty$.
\end{enumerate}
Then the function $u_y:S\to (-\infty,\infty]$ defined by  
\begin{equation}\label{eq:uQ}
u_y(x):=Q_{L,\mu,S}(x,y),\quad x\in S,
\end{equation}
has a unique minimizer $z\in S$ (and $u_y(z)$ is a real number). In particular, if $S\cap U\neq \emptyset$ and $b$ is uniformly convex on $S$, then for each $y\in S\cap U$, the function $u_y$ from \beqref{eq:uQ} has a unique minimizer $z\in S$. 
\end{lem}
\begin{proof}
We start by showing the existence of a minimizer. Since $u_y$ is convex and lower semicontinuous (because it is a sum of convex and lower semicontinuous functions), it is lower semicontinuous with respect to the weak topology \cite[Corollary 3.9, p. 61]{Brezis2011book}. Suppose first that Condition \beqref{item:SisBounded} holds, namely $S$ is bounded. Then the existence of a minimizer follows from  \cite[p. 11]{Brezis2011book} since $S$ is bounded, closed and convex, and hence compact with respect to the weak  topology (because $X$ is reflexive \cite[Corollary 3.22, p. 71]{Brezis2011book}), and $u_y$ is lower semicontinuous with respect to  the weak topology. 

Consider now the case of Condition \beqref{item:SisUnboundedRelativeUC}. Since $S\subseteq \dom(b)$ and $y\in U$, the assumed uniform convexity of $b$ relative to $(S,\{y\})$ with a relative gauge $\psi$ implies, according to \cite[Proposition 4.13(I)]{ReemReichDePierro2019jour(BregmanEntropy)}, that $B(x,y)\geq \psi(\|x-y\|)$ for all $x\in S$. Since we assume that $\lim_{t\to\infty}\psi(t)/t=\infty$ and that $S$ is unbounded, we have
\begin{equation*}
\frac{B(x,y)}{\|x\|}\geq \frac{\psi(\|x-y\|)}{\|x-y\|}\cdot\frac{\|x-y\|}{\|x\|}\xrightarrow[\|x\|\to \infty,x\in S]{} \infty.
\end{equation*}
This fact, \beqref{eq:Q_L-ISTA}, \beqref{eq:uQ} and Lemma \bref{lem:SuperCoercive} imply that $\lim_{\|x\|\to\infty, x\in S}u_y(x)/\|x\|=\infty$ and hence, in particular,  $\lim_{\|x\|\to\infty, x\in S}u_y(x)=\infty$. Therefore we can use  \cite[Corollary 3.23, p. 71]{Brezis2011book} to conclude that $u_y$ has a minimizer $z\in S$. To see that $u_y(z)$ is a real number, one observes that  the ranges of $B$, $f$ and $g$ do not include $-\infty$ and hence $-\infty<u_y(z)$. In addition, since $g$ is proper, there is $x\in S$ satisfying  $g(x)<\infty$. This fact, as well as \beqref{eq:Q_L-ISTA}, \beqref{eq:uQ}, the fact that  $S\subseteq \dom(b)$, and the fact that $z$ is a minimizer of $u_y$, imply that $u_y(z)\leq u_y(x)<\infty$. 

We now turn to proving the uniqueness of the minimizer.Suppose to the contrary that $u_y$ attains its minimum at two different points $x_1,x_2\in S$. Since $u_y$ is a sum of convex functions and a function which is strictly convex on $S$ (since we assume that $b$ is strictly convex on $S$), it is strictly convex on $S$. This fact and the equality $u_y(x_1)=u_y(x_2)$ imply that $u_y(0.5(x_1+x_2))<0.5u_y(x_1)+0.5u_y(x_2)=u_y(x_1)$, a contradiction to the assumption that $u_y(x_1)$ is the smallest value of $u_y$ on $S$. 

Finally, assume that $S\cap U\neq \emptyset$ and that $b$ is uniformly convex on $S$. If $S$ is bounded, then we are in the case of Condition \beqref{item:SisBounded} and the assertion follows from previous paragraphs. Assume now that $S$ is unbounded. We observe that the modulus of convexity $\psi_{b,S}$ of $b$ is a relative gauge on $(S,\{y\})$ for all $y\in S\cap U$ (this is true even if $S$ is bounded). Now fix some $y\in S\cap U$. Since $S$ is unbounded, we can apply \cite[Lemma 3.3]{ReemReichDePierro2019jour(BregmanEntropy)} which implies that $\lim_{t\to\infty}\psi_{b,S}(t)/t=\infty$. Thus we can conclude from previous paragraphs that $u_y$ has a unique minimizer on $S$. 
\end{proof}

The next lemma generalizes \cite[Lemma 5.2]{ReemDe-Pierro2017jour} to the setting of Bregman divergences (\cite[Lemma 5.2]{ReemDe-Pierro2017jour} by itself extends \cite[Lemma 2.2]{BeckTeboulle2009jour}).

\begin{lem}\label{lem:Optimality}
Consider the setting of Lemma \bref{lem:UniqueMinimzerQ} and let $\tilde{g}:X\to(-\infty,\infty]$ be defined by $\tilde{g}(x):=g(x)$ whenever $x\in S$ and by $\tilde{g}(x):=\infty$ whenever $x\notin S$. Assume that  $S\cap U\neq\emptyset$ and that there exists a point in $\dom(g)$ at which either $b$ or $g$ are continuous. Then an element $z\in S\cap U$ is a  minimizer of $u_y$ in $S$ if and only if there exists $\gamma\in \partial \tilde{g}(z)$ such that 
\begin{equation}\label{eq:z_gamma}
f'(y)+\gamma=\frac{L}{\mu}(b'(y)-b'(z)).
\end{equation}
\end{lem}
\begin{proof}
Let $\tilde{u}_y:X\to(-\infty,\infty]$ be defined by $\tilde{u}_y(x):=u_y(x)$ when $x\in S$ and  $\tilde{u}_y(x):=\infty$ otherwise, where $u_y$ is defined in \beqref{eq:uQ}. By combining the definition of $\tilde{u}_y$, \beqref{eq:Q_L-ISTA}, \beqref{eq:uQ}, the convexity and closedness of $S\neq\emptyset$, and the fact that $g$ is proper on $S$ and $b$ is finite on $S$, it follows that $\tilde{u}_y$ is proper, convex and lower semicontinuous. Since we assume the setting of Lemma \bref{lem:UniqueMinimzerQ}, we know by this lemma  that $u_y$ has a unique minimizer $z_{L,\mu,S,y}\in S$ and $u_y(z_{L,\mu,S,y})\in \R$. Since $\tilde{u}_y(x)=\infty$ whenever $x\notin S$, we conclude that $z_{L,\mu,S,y}$ is the unique global minimizer of $\tilde{u}_y$ on $X$.  Since $z_{L,\mu,S,y}$ is a minimizer of $\tilde{u}_y$, it follows from Fermat's rule (namely the very simple necessary and sufficient condition for minimality \cite[p. 96]{VanTiel1984book}) that $0\in \partial \tilde{u}_y(z_{L,\mu,S,y})$.

Consider the function $\tilde{v}_y:X\to(-\infty,\infty]$ which is defined for all $x\in S$ by $\tilde{v}_y(x):=f(y)+\langle  f'(y),x-y\rangle+(L/\mu)B(x,y)$ and by $\tilde{v}_y(x):=\infty$ otherwise. From this definition and because of the assumptions on $b$, it follows that $\tilde{v}_y$ is proper and convex (on $X$) and G\^ateaux differentiable in $S\cap U$ (it is also lower semicontinuous on $X$, but we do not need this property). Hence from  \cite[Theorem 5.37, p. 77]{VanTiel1984book} it follows that  $\partial \tilde{v}_y(x)=\{\tilde{v}_y'(x)\}$ for each $x\in S\cap U$. This equality and  \beqref{eq:BregmanDistance} imply that 
\begin{equation}\label{eq:v_tilde_z}
\partial \tilde{v}_y(x)=\{\tilde{v}_y'(x)\}=\{f'(y)+(L/\mu)(b'(x)-b'(y))\}, \quad\forall x\in S\cap U. 
\end{equation}
 Assume now that an element $z\in S\cap U$ satisfies $z=z_{L,\mu,S,y}$.  Since $\tilde{u}_y=\tilde{v}_y+\tilde{g}$, it follows that $\tilde{u}_y$ is the sum of two convex and proper functions. Since, according to the assumption in the formulation of this lemma, there exists a point in $\dom(g)$ at which either $b$ or $g$ are continuous, and since $\dom(g)\subseteq S\subseteq \dom(b)$ and hence $\dom(b)\cap \dom(g)=\dom(g)$, it follows that either $b$ or $g$ are continuous at some point in $\dom(b)\cap \dom(g)$. Since we already know that $0\in \partial \tilde{u}_y(z_{L,\mu,S,y})$, we conclude from the sum rule \cite[Theorem 5.38, p. 77]{VanTiel1984book} and its proof  \cite[Theorem 5.38, pp.  78-79]{VanTiel1984book} that $\partial \tilde{g}(z)\neq \emptyset$ and 
\begin{equation}\label{eq:SumRule_uvg}
0\in \partial \tilde{u}_y(z)=\partial \tilde{v}_y(z)+\partial \tilde{g}(z). 
\end{equation}
By combining this inclusion with the inclusion $z\in S\cap U$ and with \beqref{eq:v_tilde_z}, we obtain \beqref{eq:z_gamma} for some $\gamma\in \partial\tilde{g}(z)$. On the other hand, if \beqref{eq:z_gamma} holds for some $z\in S\cap U$ and  $\gamma\in \partial \tilde{g}(z)$, then (obviously) $\partial \tilde{g}(z)\neq \emptyset$. Since $z\in S\cap U$, we know from  the lines which precede \beqref{eq:v_tilde_z} that $\partial \tilde{v}_y(z)=\{\tilde{v}_y'(z)\}$. This fact, the equality $\tilde{u}_y(x)=\tilde{v}_y(x)+\tilde{g}(x)$ for all $x\in X$, the fact that $\tilde{v}_y$ and $\tilde{g}$ are proper and convex, and the sum rule  \cite[Theorem 5.38, p.  77]{VanTiel1984book}, all of these facts  show that  $\partial \tilde{v}_y(z)+\partial \tilde{g}(z)\in \partial \tilde{u}_y(z)$. By combining this inclusion with \beqref{eq:v_tilde_z} (in which we take $x:=z$) and with \beqref{eq:z_gamma}, we have $0\in \partial \tilde{u}_y(z)$. From Fermat's rule it follows that $z$ is a global minimizer of $\tilde{u}_y$ on $X$. Since $z\in S$ and since $\tilde{u}_y$ coincides with $u_y$ on $S$, it follows that $z$ is a minimizer of $u_y$ on $S$, as required. As a matter of fact, Lemma \bref{lem:UniqueMinimzerQ} implies that  $z$ coincides with $z_{L,\mu,S,y}$ because $z_{L,\mu,S,y}$  is the unique minimizer of $u_y$  on $S$. 
\end{proof}

\begin{remark}\label{rem:SufficientConditions dom(g)dom(b)}
A crucial assumption in Lemma \bref{lem:Optimality} is the existence of a point in $\dom(g)$ at which either $b$ or $g$ are continuous. 
Hence it is of interest to present several rather practical sufficient conditions which ensure the existence of such a point. As we show below, each one of the following conditions achieves this goal: 
\begin{enumerate}[(i)]
\item\label{cond:g is cont on dom(g)} $g$ is continuous on $\dom(g)$,
\item\label{cond:b is cont on dom(g)} $b$ is continuous on $\dom(g)$,
\item\label{cond:b is cont on dom(b)} $b$ is continuous on $\dom(b)$,
\item\label{cond:dom(g)U} $\dom(g)\cap U\neq\emptyset$,
\item\label{cond:SU} $S\subseteq U$.
\end{enumerate}
Indeed, Conditions \beqref{cond:g is cont on dom(g)}--\beqref{cond:b is cont on dom(g)}  immediately imply  the required assumption and Condition \beqref{cond:b is cont on dom(b)} is a particular case of Condition \beqref{cond:b is cont on dom(g)} because $\dom(g)\subseteq S\subseteq \dom(b)$.  Assume now that Condition \beqref{cond:dom(g)U} holds. Since we assume that $b$ is lower semicontinuous and convex, that $X$ is a Banach space, and that $U:=\Int(\dom(b))\neq \emptyset$, we can apply \cite[Proposition 3.3, p. 39]{Phelps1993book_prep} to conclude that $b$ is continuous on $U$. In particular, $b$ is continuous at the point in $U$ which belongs to $\dom(g)$, namely the assertion follows. Condition \beqref{cond:SU} is just a particular case of Condition \beqref{cond:dom(g)U} because $\dom(g)\subseteq S$ and $g$ is proper on $S$. 
\end{remark}

The following lemma generalizes \cite[Lemma 2.3]{BeckTeboulle2009jour} (finite-dimensional Euclidean spaces) and \cite[Lemma 3.1]{ReemDe-Pierro2017jour} (real Hilbert spaces) to our setting of real reflexive Banach spaces and Bregman divergences. 
\begin{lem}\label{lem:FB_k}
Consider the setting of Lemma \bref{lem:Optimality} and let  $F(x):=f(x)+g(x)$ for each  $x\in S$. Suppose that for some $y\in U$, the minimizer $z$ of $u_y$ (from \beqref{eq:uQ})  satisfies $z\in S\cap U$ and 
\begin{equation}\label{eq:F<Q} 
F(z)\leq Q_{L,\mu,S}(z,y). 
\end{equation}
Then for all $x\in S$, 
\begin{equation}\label{eq:FB_k}
F(x)-F(z)\geq \frac{L}{\mu}\left(B(x,z)-B(x,y)\right).
\end{equation}
\end{lem}
\begin{proof}
Since $u_y(z)=Q_{L,\mu,S}(z,y)\in\R$ (according to Lemma \bref{lem:Optimality}), since the ranges of $f$ and $g$ do not include $-\infty$ and since \beqref{eq:F<Q} holds, we have $F(z)\in\R$. Let $\tilde{f}:X\to (-\infty,\infty]$ be the function defined by $\tilde{f}(x):=f(x)$ if $x\in \dom(b)$ and $\tilde{f}(x):=\infty$ otherwise. The convexity of $\dom(b)$ and of $f$ on $\dom(b)$ imply that $\tilde{f}$ is convex on $X$. Since $f$ is G\^ateaux differentiable in the open subset $U$ and since $\tilde{f}(x)=f(x)$ for each $x\in U$, the assumption that $y\in U$ and \cite[Theorem 5.37, p. 77]{VanTiel1984book} imply that $\{f'(y)\}=\{\tilde{f}'(y)\}=\partial \tilde{f}(y)$. Let $\gamma$ be the vector in $\partial\tilde{g}(z)$ which satisfies \beqref{eq:z_gamma}, the existence of which is ensured by Lemma  \bref{lem:Optimality}. Since $S\subseteq \dom(b)$, the subgradient inequality and the equalities $\tilde{f}(x)=f(x)$ and $\tilde{g}(x)=g(x)$ for all $x\in S$, imply that 
\begin{equation}\label{eq:f_subgrad}
f(x)\geq f(y)+\langle f'(y),x-y\rangle,\quad  x\in S
\end{equation}
\quad\quad\quad\quad\quad\quad and
\begin{equation}\label{eq:g_subgrad}
g(x)\geq g(z)+\langle \gamma,x-z\rangle,\quad x\in S.
\end{equation}
It follows from the equality $F(x)=f(x)+g(x)$ for every $x\in S$, from \beqref{eq:F<Q}, \beqref{eq:f_subgrad}, \beqref{eq:g_subgrad}, from 
\beqref{eq:Q_L-ISTA} with $z$ instead of $x$, from Lemma \bref{lem:Optimality}, and from  \beqref{eq:BregmanDistance}, that for all $x\in S$, we have
\begin{multline*}
F(x)-F(z)\geq (f(x)+g(x))-Q_{L,\mu,S}(z,y)\\
\geq f(y)+\langle f'(y),x-y\rangle+g(z)+\langle \gamma,x-z\rangle-Q_{L,\mu,S}(z,y)\\
=\langle f'(y)+\gamma,x-z\rangle-(L/\mu)B(z,y)\\
=\langle (L/\mu)(b'(y)-b'(z)),x-z\rangle-(L/\mu)B(z,y)\\
=\langle (L/\mu)b'(y),x-z\rangle-\langle (L/\mu)b'(z),x-z\rangle 
-(L/\mu)\big(b(z)-b(y)-\langle b'(y),z-y\rangle\big)\\
=(L/\mu)\big(\langle b'(y),x-y\rangle +b(y)\big)-(L/\mu)\big(\langle b'(z), x-z\rangle+b(z)\big)\\
=(L/\mu)(b(x)-B(x,y))-(L/\mu)(b(x)-B(x,z))=(L/\mu)(B(x,z)-B(x,y)),
\end{multline*}
as claimed.
\end{proof}

The next lemma is  needed in order to ensure simple sufficient conditions for the  minimizer of $u_y$ from \beqref{eq:uQ} to belong to $S\cap U$. These conditions are useful for Assumption \bref{assum:pInU}.

\begin{lem}\label{lem:zInSU}
Consider the setting of Lemma \bref{lem:Optimality} and suppose in addition that at least one of the following conditions holds:
\begin{enumerate}[(A)]
\item\label{item:SU} $S\subseteq U$, 
\item\label{item:NonempetySubdif} For each $x\in S\cap \partial U$, either 
$\partial b(x)=\emptyset$ or $\partial \tilde{g}(x)=\emptyset$, where $\partial U$ is the boundary of $U$ and $\tilde{g}:X\to (-\infty,\infty]$ is defined as $\tilde{g}(x):=g(x)$ if $x\in S$ and $\tilde{g}(x):=\infty$ otherwise.
\end{enumerate}
Then the minimizer $z$ of $u_y$ (from \beqref{eq:uQ}) satisfies $z\in S\cap U$. In particular, if the assumptions mentioned in Lemma \bref{lem:Optimality} hold, $X$ is finite-dimensional and $b$ is essentially smooth on $U$, then $z\in S\cap U$. 
\end{lem}
\begin{proof}
The case of Condition \beqref{item:SU} is obvious, and so from now on assume that Condition  \beqref{item:NonempetySubdif} holds.  In what follows we suppose to the contrary that $z\in S\cap \partial U$.

Let $\tilde{u}_y:X\to (-\infty,\infty]$ be defined by $\tilde{u}_y(x):=u_y(x)$ when $x\in S$ (where $u_y$ is defined in \beqref{eq:uQ}), and $\tilde{u}_y(x):=\infty$ otherwise. We can write 
\begin{equation}
\tilde{u}_y(x):=f(y)+\langle f'(y),x-y\rangle+\frac{L}{\mu}B(x,y)+\tilde{g}(x),\quad\forall x\in X.
\end{equation}
 Then $\tilde{u}_y$ is convex, proper and lower semicontinuous. 
We know from Lemma \bref{lem:UniqueMinimzerQ} that $z$ is the unique global minimizer of $u_y$ on $S$. Since $\tilde{u}_y(x)=\infty$ when $x\notin S$, it follows that $z$ is also the global minimizer of $\tilde{u}_y$ on $X$. Thus the simple necessary and sufficient condition for minimizers  \cite[p. 96]{VanTiel1984book} (Fermat's  rule) implies that $0\in \partial \tilde{u}_y(z)$. Since  for all $y\in U$ and all $x\in X$, we have 
\begin{equation*} 
\tilde{u}_y(x)=f(y)+\langle f'(y),x-y\rangle+\tilde{g}(x)+(L/\mu)(b(x)-b(y)-\langle b'(y),x-y\rangle),
\end{equation*} 
we see that $\tilde{u}_y$ is the sum of the two convex and proper functions defined by  $\tilde{u}_{1,y}(x):=f(y)+\langle f'(y),x-y\rangle-(L/\mu)(b(y)+\langle b'(y),x-y\rangle)+\tilde{g}(x)$ and $\tilde{u}_{2,y}(x):=(L/\mu)b(x)$ for each $x\in X$. 

We have $\dom(\tilde{u}_{1,y})=\dom(g)$ and $\dom(\tilde{u}_{2,y})=\dom(b)$. Moreover, the set of points at which $\tilde{u}_{1,y}$ is finite and continuous coincides with the set of points at which $\tilde{g}$ is finite and continuous, namely with the set of points at which $g$ is continuous. In addition, the set of points at which $\tilde{u}_{2,y}$ is finite and continuous coincides with the set of points at which $b$ is finite and continuous. Since we assume that the setting of Lemma \bref{lem:Optimality} holds, there exists some point in $\dom(b)\cap \dom(g)$ at which either $b$ or $g$ are continuous. Hence we can apply the sum rule \cite[Theorem 5.38, p. 77]{VanTiel1984book} and its proof  \cite[Theorem 5.38, pp.  78-79]{VanTiel1984book} to conclude that $\partial \tilde{u}_y(x)=\partial \tilde{u}_{1,y}(x)+\partial \tilde{u}_{2,y}(x)$ for every $x\in X$ and, furthermore, that the subsets  $\partial \tilde{u}_{1,y}(x)$ and $\partial \tilde{u}_{2,y}(x)$ are nonempty whenever $\partial \tilde{u}_y(x)$ is nonempty. 

Since we already know that $0\in\partial\tilde{u}_y(z)$, it follows that both $\partial \tilde{u}_{1,y}(z)\neq\emptyset$ and $\partial \tilde{u}_{2,y}(z)\neq\emptyset$. Therefore, by direct computation (or by the sum rule), $\partial \tilde{g}(z)=\partial \tilde{u}_{1,y}(z)-f'(y)+(L/\mu)b'(y)\neq\emptyset$ and $\partial b(z)=(\mu/L)\partial \tilde{u}_{2,y}(z)\neq \emptyset$. These relations contradict Condition  \beqref{item:NonempetySubdif}. 

The preceding discussion proves that $z$ cannot be in $S\cap \partial U$. Since obviously $z$, which is the minimizer of $u_y$ on $S$, belongs to $S$, it follows that $z\notin \partial U$. But $S\subseteq \dom(b)\subseteq \cl{\dom(b)}=\cl{U}$ (the last equality is just the well-known fact that says that the closure of the nonempty interior of a convex subset is equal to the closure of the subset itself \cite[Theorem 2.27(b), p. 29]{VanTiel1984book}), and hence $z\in S\subseteq \cl{U}$. This fact and the equality $\cl{U}=U\cup \partial U$ imply that $z\in U$. Hence $z\in S\cap U$, as required. 

 It remains to show that $z\in S\cap U$ in the particular case where $X$ is finite-dimensional and  $b$ is essentially smooth. Under these assumptions we can apply \cite[Theorem 26.1. pp. 251--252]{Rockafellar1970book} which implies that $\partial b(x)=\emptyset$ for each boundary point $x$ of $U$. Thus Condition \beqref{item:NonempetySubdif} holds and the assertion follows from previous paragraphs. 
\end{proof}

\begin{remark}
In view of the results of \cite[Section 5]{BauschkeBorweinCombettes2001jour}, mainly the ones related to essentially smooth  functions defined on reflexive Banach spaces, it might be that the finite-dimensional sufficient condition mentioned in Lemma \bref{lem:zInSU} can be extended in one way or another to a class of functions defined on some infinite-dimensional spaces. 
\end{remark}

The following lemma is needed for proving that the proximal forward-backwards algorithms in Section \bref{sec:ISTA-Bregman} are well defined (see Remark \bref{rem:L_kFiniteFx_k ISTA}). Versions of it are known in more restricted settings \cite[Lemma 1.2.3, pp. 22--23]{Nesterov2004book}, \cite[Lemma 5.1]{ReemDe-Pierro2017jour}, \cite[Lemma 2.1]{Cohen1980jour}. For the sake of completeness, we provide the proof in the appendix (Section 
\bref{sec:Appendix}). 

\begin{lem}\label{lem:LipschitzUpperBound}
Let $(X,\|\cdot\|)$ be a real normed space and let $U$ be an open subset of $X$. 
Suppose  that $x,y\in U$ are given and that the line segment $[x,y]$ is 
contained in $U$. Let $f:U\to \R$ be a continuously Fr\'echet  differentiable function the derivative $f'$ of which is Lipschitz continuous along the line segment $[x,y]$ 
 with a Lipschitz constant $L(f',[x,y])\geq 0$. Then the following inequality is satisfied for all $L\geq L(f',[x,y])$: 
\begin{equation}\label{eq:fL}
f(x)\leq f(y)+\langle f'(y), x-y\rangle+\frac{1}{2}L\|x-y\|^2. 
\end{equation}
\end{lem}

The following proposition is needed for the proof of Corollary \bref{cor:AlmostO(1/k)}. Its proof can be  found in \cite{ReemReichDePierro2019arXiv(StabilityOptimalV1)}.

\begin{prop}\label{prop:f'Lip}
Suppose that $f:U\to \R$ is a twice continuously (Fr\'echet) differentiable function defined on an open and convex subset $U$  of some real normed space $(X,\|\cdot\|)$, $X\neq\{0\}$. Suppose that $C$ is a convex subset of $X$ which has the property that $C\cap U\neq\emptyset$. Assume that $f''$ is bounded and uniformly continuous on  bounded subsets of $C\cap U$. Fix an arbitrary $y_0\in C\cap U$, and let $s:=\sup\{\|f''(x)\|: x\in C\cap U\}$ and $s_0:=\|f''(y_0)\|$. If $s=\infty$, then for each strictly increasing sequence $(\lambda_k)_{k=1}^{\infty}$ of positive numbers which satisfies $\lambda_1>s_0$ and $\lim_{k\to\infty}\lambda_k=\infty$, there exists an increasing sequence $(S_k)_{k=1}^{\infty}$ of bounded and convex subsets of $C$ (and also closed if $C$ is closed) such that $S_k\cap U\neq\emptyset$ for all $k\in \N$, that $\cup_{k=1}^{\infty}S_k=C$, and that for each $k\in\N$, the function $f'$ is  Lipschitz continuous on $S_k\cap U$ with $\lambda_k$ as a Lipschitz constant; moreover, if $C$ contains more than one point, then also $S_k\cap U$ contains more than one point for each $k\in\N$. Finally, if $s<\infty$, then $f'$ is Lipschitz continuous on $C\cap U$ with $s$ as a Lipschitz constant. 
\end{prop}

Now it is possible to prove Theorem \bref{thm:BISTA-Convergence} and the corollaries which follow it. The proofs are based on the previous assertions  and also on the notation and assumptions of Section  \bref{sec:ISTA-Bregman}. \\

{\noindent \bf Proof of Theorem \bref{thm:BISTA-Convergence}:}
 Since $\cup_{k=1}^{\infty}S_k=C$ and since $x_{\textnormal{min}}\in \textnormal{MIN}(F)\subseteq C$, there exists an index $k_0\in \N$ such that $x_{\textnormal{min}}\in S_{k_0}$. Since  $S_{k}\subseteq S_{k+1}$ for all $k\in\N$ (Assumption \bref{assum:MonotoneStronglyConvex}), one has $x_{\textnormal{min}}\in S_{k}$ for all $k\geq k_0$. For a technical reason (see the discussion after \beqref{eq:sumFx_kx_iB_1}), if $F(x_1)=\infty$, then  we take $k_0$ to be at least 2. This is  possible since by definition, $k_0$ is just an index (not necessarily the first) for which $x_{\textnormal{min}}\in S_{k_0}$, and hence, if $x_{\textnormal{min}}\in S_1$, then also $x_{\textnormal{min}}\in S_2$ by the inclusion $S_1\subseteq S_2$. 
 
 From Remark \bref{rem:L_kFiniteFx_k ISTA} it follows that for all $i\in\N$ inequality \beqref{eq:F<Q}  holds with $S:=S_{i+1}$, $L:=L_{i+1}$, $y:=x_{i}$, $\mu:=\mu_{i+1}$ and $z:=x_{i+1}$,  and also that  $F(x_{i+1})$ is finite for all $i\in\N$. 
According to our assumption $b$ is strongly convex on $S_k$ for all $k\in\N$. Moreover, $x_{\textnormal{min}}\in S_k$ for all $k\geq k_0$ and $x_k\in S_k\cap U$ for each $k\in\N$ (as follows from Assumption \bref{assum:pInU} and an induction argument). Thus we can use  Lemma \bref{lem:FB_k}, where in \beqref{eq:FB_k} we substitute $x:=x_{\textnormal{min}}$, $y:=x_{i}$, $S:=S_{i+1}$, $\mu:=\mu_{i+1}$, $L:=L_{i+1}$ and $z:=x_{i+1}$, where $k_0\leq i\leq k$. This yields 
\begin{equation}\label{eq:F>B}
F(x_{\textnormal{min}})-F(x_{i+1})\geq (L_{i+1}/\mu_{i+1})(B(x_{\textnormal{min}},x_{i+1})-B(x_{\textnormal{min}},x_{i})).
\end{equation}
Since $F(x_{\textnormal{min}})-F(x_{i+1})\leq 0$, it follows from \beqref{eq:F>B} that  $B(x_{\textnormal{min}},x_{i+1})-B(x_{\textnormal{min}},x_{i})\leq 0$. 
From \beqref{eq:tau_rho ISTA}, \beqref{eq:F>B}, the inequality  $F(x_{\textnormal{min}})-F(x_{i+1})\leq 0$, and the inequalities $\mu_{i+1}\geq \mu_{k+1}$ (Assumption \bref{assum:MonotoneStronglyConvex}) and 
$\tau_{i+1}\leq \tau_{k+1}$ (Remark \bref{rem:BoundLk ISTA}) for all $i\in \{k_0,\ldots,k\}$,  it follows that 
\begin{multline}\label{eq:Fx_iB_1}
\frac{F(x_{\textnormal{min}})-F(x_{i+1})}{\tau_{k+1}}\geq
\frac{F(x_{\textnormal{min}})-F(x_{i+1})}{\tau_{i+1}}\geq \frac{F(x_{\textnormal{min}})-F(x_{i+1})}{L_{i+1}}\geq\\ \frac{B(x_{\textnormal{min}},x_{i+1})-B(x_{\textnormal{min}},x_{i})}{\mu_{i+1}}
\geq \frac{B(x_{\textnormal{min}},x_{i+1})-B(x_{\textnormal{min}},x_{i})}{\mu_{k+1}}.
\end{multline}
By summing \beqref{eq:Fx_iB_1} from $i:=k_0$ to $i:=k$, we obtain 
\begin{equation}\label{eq:sumFx_kB_1}
(k+1-k_0)F(x_{\textnormal{min}})-\sum_{i=k_0}^{k}F(x_{i+1})\geq (\tau_{k+1}/\mu_{k+1})(B(x_{\textnormal{min}},x_{k+1})-B(x_{\textnormal{min}},x_{k_0})).
\end{equation}
Using  Lemma \bref{lem:FB_k}, where in \beqref{eq:FB_k} we substitute $x:=x_{i}$, $y:=x_i$, $S:=S_{i+1}$, $L:=L_{i+1}$, $i\geq k_0$, $z:=x_{i+1}$, and $\mu:=\mu_{i+1}$,  using \beqref{eq:tau_rho ISTA} and the nonnegativity of $B$, and the fact that $B(x_i,x_i)=0$, and using the inequality $\mu_{k_0}\geq \mu_{i+1}$ for all $i\in\{k_0,\ldots,k\}$ (as a result of Assumption \bref{assum:MonotoneStronglyConvex}), we obtain 
\begin{equation}\label{eq:Fx_i_i+1B_1}
F(x_{i})-F(x_{i+1})\geq (L_{i+1}/\mu_{i+1})B(x_{i},x_{i+1})\geq (L_{i+1}/\mu_{k_0})B(x_{i},x_{i+1}).
\end{equation}
After multiplying \beqref{eq:Fx_i_i+1B_1} by $i-k_0$, summing from $i:=k_0$ to $i:=k$, and 
performing simple manipulations, we arrive at 
\begin{multline}\label{eq:sumFx_kx_iB_1}
-(k+1-k_0)F(x_{k+1})+\sum_{i=k_0}^{k}F(x_{i+1})\\
=\sum_{i=k_0}^{k}\left((i-k_0)F(x_{i})-(i+1-k_0)F(x_{i+1})+F(x_{i+1})\right)\\
\geq \sum_{i=k_0}^{k}(L_{i+1}/\mu_{k_0})(i-k_0)B(x_{i},x_{i+1}).
\end{multline}
There is a minor issue related to \beqref{eq:sumFx_kx_iB_1} which should be noted: if $k_0=1$ and $F(x_1)=\infty$, then one of the terms in \beqref{eq:sumFx_kx_iB_1} is $0\cdot\infty$, and therefore it is not defined. In order to avoid this possibility, we re-defined $k_0$ in advance (see the  beginning of the proof) to be 2 (the only case where $F(x_k)=\infty$ is when $k=1$, since Remark \bref{rem:L_kFiniteFx_k ISTA} ensures that $F(x_k)\in\R$ for all $k\geq 2$). 

After summing \beqref{eq:sumFx_kB_1} and \beqref{eq:sumFx_kx_iB_1} and using the nonnegativity of some terms, it follows that 
\begin{multline}
(k+1-k_0)(F(x_{\textnormal{min}})-F(x_{k+1}))\\
\geq (\tau_{k+1}/\mu_{k+1})(B(x_{\textnormal{min}},x_{k+1})-B(x_{\textnormal{min}},x_{k_0}))+\sum_{i=k_0}^{k}(i-k_0)(L_{i+1}/\mu_{k_0})B(x_{i},x_{i+1})\\
\geq -(\tau_{k+1}/\mu_{k+1}) B(x_{\textnormal{min}},x_{k_0}).
\end{multline}
This inequality implies \beqref{eq:O(1/k)}, as claimed.

Now, if \beqref{eq:tau_mu_k} holds, then obviously $\lim_{k\to\infty}F(x_k)=F(x_{\textnormal{min}})$. From  now on we assume that \beqref{eq:tau_mu_k} holds, that $B$ has the limiting difference property and that all of its first type level-sets are bounded. Our goal is to prove that the weak limit  $\lim_{k\to\infty}x_k$ exists and belongs to $\textnormal{MIN}(F)$. Since $F(x_{\textnormal{min}})-F(x_k)\leq 0$  for all $k\in \N$, it follows from \beqref{eq:Fx_iB_1} that \begin{equation}\label{eq:BregmanMonotone}
B(x_{\textnormal{min}},x_{k+1})\leq B(x_{\textnormal{min}},x_k), \quad \forall  k\geq k_0,\,k\in \N, 
\end{equation}
namely, the sequence of nonnegative numbers  $(B(x_{\textnormal{min}},x_k))_{k=k_0}^{\infty}$ is 
decreasing. In particular, one has $B(x_{\textnormal{min}},x_k)\leq B(x_{\textnormal{min}},x_{k_0})$ for every $k\geq k_0$. This inequality and the nonnegativity of $B$ imply that $(B(x_{\textnormal{min}},x_k))_{k=1}^{\infty}$ is bounded. By our assumption the first type level-set $\{y\in U: B(x_{\textnormal{min}},y)\leq B(x_{\textnormal{min}},x_{k_0})\}$ is bounded. Since $(x_k)_{k=k_0}^{\infty}$ is contained in this level-set, $(x_k)_{k=k_0}^{\infty}$ and hence $(x_k)_{k=1}^{\infty}$ are bounded.

Let $x_{\infty}$ be an arbitrary weak cluster point of $(x_k)_{k\in \N}$. Such a cluster point exists because $X$ is reflexive and $(x_k)_{k\in \N}$ is bounded. Since $C$ is closed and convex, and $X$ is reflexive, $C$ is  weakly closed and hence $x_{\infty}\in C$. Now we observe that \beqref{eq:tau_mu_k} implies that 
\begin{equation}
\frac{\tau_{k+1}}{(k+1-k_0)\mu_{k+1}}=\frac{\tau_{k+1}}{(k+1)\mu_{k+1}}\cdot\frac{k+1}{k+1-k_0}
\underset{k\to \infty}{\longrightarrow} 0,
\end{equation}
next we go to a subsequence which converges weakly to $x_{\infty}$, and then we recall that $F$ is convex and  lower semicontinuous (with respect to the norm topology), and hence weakly lower semicontinuous \cite[Corollary 3.9, p. 61]{Brezis2011book}. These observations and \beqref{eq:O(1/k)} imply  that $F(x_{\infty})\leq F(x_{\textnormal{min}})$. This proves that $F(x_{\infty})=F(x_{\textnormal{min}})$ because $F(x_{\textnormal{min}})$ is the minimal value of $F$ and hence 
$F(x_{\textnormal{min}})\leq F(x_{\infty})$ holds trivially. 

From \beqref{eq:BregmanMonotone} it follows that $\lim_{k\to\infty}B(x_{\textnormal{min}},x_k)$ exists (and  is finite). Since the previous paragraph shows that  $x_{\infty}\in \textnormal{MIN}(F)$, it follows from Assumption \bref{assum:O(F)} that $x_{\infty}\in U$. Moreover, since $x_{\infty}\in C$ and $C=\cup_{i=1}^{\infty}S_i$, there exists an index $i_0\in\N$ such that $x_{\infty}\in S_{i_0}$. Assumption \bref{assum:MonotoneStronglyConvex} implies that $x_{\infty}\in S_i$ for all $i\geq i_0$. By using  Lemma \bref{lem:FB_k}, where in \beqref{eq:FB_k} we take any  $i_0\leq i\in\N$ and substitute $S:=S_{i+1}$, $\mu:=\mu_{i+1}$, $L:=L_{i+1}$, $z:=x_{i+1}$, $x:=x_{\infty}\in S_{i+1}$ and $y:=x_{i}$, we have
\begin{equation*}
F(x_{\infty})-F(x_{i+1})\geq (L_{i+1}/\mu_{i+1})(B(x_{\infty},x_{i+1})-B(x_{\infty},x_{i})).
\end{equation*}
Since $F(x_{\infty})=F(x_{\textnormal{min}})\leq F(x_{i+1})$ as shown a few lines above, one has $B(x_{\infty},x_{i+1})\leq B(x_{\infty},x_{i})$ for each $i\geq i_0$. Therefore $(B(x_{\infty},x_{i}))_{i=i_0}^{\infty}$ is a decreasing sequence of nonnegative numbers and, as a result, $\lim_{i\to\infty}B(x_{\infty},x_i)$ exists. Since $x_{\infty}$ was an arbitrary weak cluster point of $(x_k)_{k=1}^{\infty}$, since we assume that $B$ has the limiting difference property and since $(x_k)_{k=1}^{\infty}$ is a bounded sequence in $U$, Lemma \bref{lem:BregmanLimit} ensures that $(x_k)_{k=1}^{\infty}$  converges weakly to a point $z_{\infty}\in U$. From previous paragraphs we have $z_{\infty}\in\textnormal{MIN}(F)$. 

Finally, if $b'$ is weak-to-weak$^*$ sequentially continuous and either $U$ is bounded or for each $x\in C$, there exists $r_x\geq 0$ such that $\{y\in U: \|y\|\geq r_x\}\neq\emptyset$ and $b$ is  uniformly convex relative to $(\{x\},\{y\in U: \|y\|\geq r_x\})$ with a gauge $\psi_x$ which satisfies $\lim_{t\to\infty}\psi_x(t)=\infty$, then Remarks \bref{rem:BregmanDivergence}\beqref{item:LimitingDifferenceProperty} and \bref{rem:BregmanDivergence}\beqref{item:LevelSetBounded} imply that $B$ has the limiting difference property  and its first type level-sets are bounded. Hence the previous paragraph implies that $(x_k)_{k=1}^{\infty}$  converges weakly to a point $z_{\infty}\in \textnormal{MIN}(F)$.
\qed\vspace{0.3cm}

{\bf \noindent Proof of Corollary \bref{cor:AlmostO(1/k)}:}
Suppose first that $\sup\{\|f''(x)\|: x\in C\cap U\}=\infty$. 
 Fix some $y_0\in C\cap U$ and let $(\lambda_k)_{k=1}^{\infty}$ be any strictly  increasing sequence of positive numbers which satisfies $\lambda_1>\|f''(y_0)\|$, $\lim_{k\to\infty}\lambda_k/k=0$ and  $\lim_{k\to\infty}\lambda_k=\infty$ (say, $\lambda_k=\alpha k^q$ for all $k\in\N$, where $\alpha>\|f''(y_0)\|$ and $q\in (0,1)$ are fixed). Since $\lim_{k\to\infty}\lambda_k=\infty$, $\lambda_1>\|f''(y_0)\|$ and $C$ contains more than one point, Proposition \bref{prop:f'Lip} implies that there is an increasing sequence of nonempty closed and convex subsets $S_k\subseteq C$ such that $S_k\cap U$ contains more than one point and $L(f',S_k\cap U)\leq \lambda_k$ for every $k\in\N$, and such that $\cup_{k=1}^{\infty}S_k=C$. 

Let $\mu_k:=\mu$ for all $k\in\N$ and consider two cases. In the first case, we are in the  Lipschitz step size rule (Algorithm \bref{alg:LipschitzStep}), and in the second case we are in the backtracking step size rule (Algorithm \bref{alg:BacktrackingStep}). If the first case holds, then let $L_1$ be any positive number satisfying $L_1\geq L(f',S_1)$ and let $L_k:=\lambda_k$ for each $k\geq 2$; in addition, let $\tau_k:=L_k$ for every $k\in\N$. If the second case holds, then let $L_1$ be any positive number satisfying $L_1\leq \eta L(f',S_1)$ and let $L_k:=\eta^{i_k}L_{k-1}$ whenever $k\geq 2$, where $i_k$ is defined in Algorithm \bref{alg:BacktrackingStep}; in addition, let $\tau_k:=\eta\lambda_k$ for all $k\in\N$. These choices ensure  that $\tau_{k+1}\geq \tau_k\geq L_k$ for each $k\in\N$. Indeed, in  the Lipschitz step size rule this assertion follows immediately from the definition of $\tau_k$ and the assumption that $(\lambda_j)_{j=1}^{\infty}$ is increasing; in the backtracking step size rule the assertion follows from the definition of $\tau_k$ and the fact that $L_k\leq \eta L(f',S_k)$ (as shown in Remark \bref{rem:BoundLk ISTA}; note that the proof there holds whenever $L_1\leq \eta L(f',S_1\cap U)$, no matter whether $S_j\neq C$ for some $j\in\N$ or not) and hence, from the choice of $\lambda_k$, we have $L_k\leq \eta \lambda_k=\tau_k$ for each $k\in\N$. In addition, in the first case we also have $L_k=\lambda_k\geq L(f',S_k\cap U)$, as follows from the previous paragraph. Hence we can use Theorem \bref{thm:BISTA-Convergence} which implies that \beqref{eq:O(1/k)} holds (and hence also \beqref{eq:k^q} when $\lambda_k=\alpha k^q$).

It remains to consider the case where $\sup\{\|f''(x)\|: x\in C\cap U\}<\infty$. In this case $f'$ is Lipschitz continuous on $C\cap U$ (see Proposition \bref{prop:f'Lip}). Denote $S_k:=C$ for each $k\in \N$. In the case of Algorithm \bref{alg:LipschitzStep}, take any positive number $L$ which satisfies $L\geq L(f',C\cap U)$ and denote $L_k:=L:=\tau_k$ for every $k\in\N$. In the case of Algorithm \bref{alg:BacktrackingStep}, select the $L_k$ parameters according to the rule mentioned there with $L_1>\eta L(f',C\cap U)$. Remark \bref{rem:BoundLk ISTA} ensures that $L_k=L_1$ for each $k\in\N$. Let $\tau_k:=L_1$ for all $k\in\N$. Now we can use Theorem \bref{thm:BISTA-Convergence} which implies that an $O(1/k)$ rate of convergence in the function values holds (see \beqref{eq:O(1/k)}). 

Finally, suppose that for each $x\in C$, there exists $r_x\geq 0$  such that $\{y\in U: \|y\|\geq r_x\}\neq\emptyset$ and $b$ is  uniformly convex relative to $(\{x\},\{y\in U: \|y\|\geq r_x\})$ with a gauge $\psi_x$ satisfying $\lim_{t\to\infty}\psi_x(t)=\infty$, and that $b'$ is weak-to-weak$^*$ sequentially continuous on $U$. Since \beqref{eq:tau_mu_k} holds by our choice of $\tau_k$ and $L_k$, Theorem \bref{thm:BISTA-Convergence} ensures that  $(x_k)_{k=1}^{\infty}$ converges weakly to a solution of \beqref{eq:F-ISTA}.
\qed\vspace{0.3cm}

{\noindent\bf Proof of Corollary \bref{cor:O(1/k)}:}
Denote, as stated, $S_k:=C$ and $\mu_k:=\mu$ for each $k\in \N$. The assumption on $b$ implies that $(S_k)_{k=1}^{\infty}$ is a telescopic  sequence in $C$. The assumption on $f'$ implies that $f'$ is Lipschitz continuous on each $S_k$ with a Lipschitz constant $L(f',C\cap U)$. In the case of Algorithm \bref{alg:LipschitzStep}, take any positive number $L$ which satisfies $L\geq L(f',C\cap U)$ and denote $L_k:=L:=\tau_k$ for every $k\in\N$. In the case of Algorithm \bref{alg:BacktrackingStep}, select the $L_k$ parameters according to the rule mentioned there with $L_1>\eta L(f',C\cap U)$. Remark \bref{rem:BoundLk ISTA} ensures that $L_k=L_1$ for each $k\in\N$. Let $\tau_k:=L_1$  for all $k\in\N$. In both cases we can use Theorem \bref{thm:BISTA-Convergence} which implies that an $O(1/k)$ rate of convergence in the function values holds (see \beqref{eq:O(1/k)}), namely that $F(x_k)-F^*=O(1/k)$ for all $k\geq k_0$. The choice of $k_0$ (see the first lines of proof of Theorem \bref{thm:BISTA-Convergence}) shows that we can take  $k_0=1$, unless $F(x_1)=\infty$, where  in this latter case we can take $k_0=2$. 

As for the convergence of $(x_k)_{k=1}^{\infty}$ to a solution of \beqref{eq:F-ISTA}, we separate the proof into two cases, according to either the assumption that $C$ is bounded or the assumption that $b$ is uniformly convex on $\dom(b)$. In the first case the proof is similar to the proof of Theorem \bref{thm:BISTA-Convergence}, where the only difference is that we do not need to invoke any assumption on the boundedness of the first type level-sets of $B$ since this assumption was needed there only to ensure the boundedness of $(x_k)_{k=1}^{\infty}$ (see the lines after \beqref{eq:BregmanMonotone}), and here we assume in advance that $C$ is bounded and, according to Assumption \bref{assum:pInU}, we know that $(x_k)_{k=1}^{\infty}$ is contained in $C$. 

In the second case, if $\dom(b)$ is bounded, then we continue as in the first case above since $C\subseteq\dom(b)$. Assume now that $\dom(b)$ is not bounded. Then $U$ is also not bounded, because otherwise $\cl{U}$ is bounded too and then the equality $\cl{\dom(b)}=\cl{U}$ implies that $\dom(b)$ is bounded, a contradiction. Now let $x\in C$ be given and let $r_x$ be an arbitrary nonnegative number. Then the set $\{y\in U: \|y\|\geq r_x\}$ is not empty (otherwise $U$ is bounded). Since $b$ is uniformly convex on $\dom(b)$, it is obviously uniformly convex relative to the pair $(\{x\},\{y\in U: \|y\|\geq r_x\})$, where the relative gauge $\psi$ is simply the modulus of uniform convexity of $b$ on $\dom(b)$ (see Definition \bref{def:TypesOfConvexity}). This gauge satisfies $\lim_{t\to\infty}\psi(t)=\infty$, as follows from \cite[Lemma 3.3]{ReemReichDePierro2019jour(BregmanEntropy)}, because $\dom(b)$ is unbounded (actually, \cite[Lemma 3.3]{ReemReichDePierro2019jour(BregmanEntropy)} implies that $\psi$ is even supercoercive). Since $b'$ is weak-to-weak$^*$ sequentially continuous, we conclude from Theorem \bref{thm:BISTA-Convergence} that $(x_k)_{k=1}^{\infty}$ converges weakly to a solution of \beqref{eq:F-ISTA}.
\qed

\section{Conclusions}\label{sec:Conclusions}
In this paper, we presented, in a rather general setting, new Bregmanian variants of the proximal gradient 
method for solving the widely useful minimization problem of the sum of two convex functions over a  constraint set. A major advantage of our method (TEPROG) is that it does not require the smooth term in the objective function to have a Lipschitz continuous gradient, an assumption which restricts the scope of applications of the proximal gradient method, but nonetheless is imposed in almost all of the many works devoted to this method. We were able to do so by decomposing the constraint set into a certain telescopic  union of subsets, and performing the minimization needed in each of the iterative steps over one subset from this union, instead of over the entire constraint set. Moreover, under practical assumptions, we were able to prove a sublinear non-asymptotic convergence of TEPROG (or, sometimes, a rate of convergence which is arbitrarily close to sublinear) to the optimal value of the objective function, as well as the weak convergence of the iterative sequence to a minimizer.  We have also obtained a few results which, we feel, are of independent interest, such as Lemma \bref{lem:FB_k} and Lemma \bref{lem:zInSU}. 

We believe that TEPROG, as well as the ``telescopic'' idea (regarding the telescopic sequence), hold a promising potential to be applied in other theoretical and practical scenarios. It will be interesting to test TEPROG numerically and to compare it with other methods, and also to check whether suitable inexact versions of TEPROG, namely ones which allow errors to appear during the iterative process, exhibit similar convergence properties (in this connection, see \cite{ReemDe-Pierro2017jour}, \cite{ReemReich2018jour} and the  references therein). 

\section{Appendix}\label{sec:Appendix}
Here we provide the proofs of some claims mentioned without proofs in previous sections. 
\begin{proof}[{\bf Proofs of some claims mentioned in Remark \bref{rem:BoundLk ISTA}}] We first show that in the backtracking step size rule, if  $S_j\neq C$ for some $j$, then $L_k\leq \eta L(f',S_k\cap U)$ for each $k\in\N$. The case $k=1$ holds by our assumption on $L_1$ since $S_j\neq C$ for some $j\in \N$. Let $k\geq 2$ and suppose, by induction,  that the claim holds for all natural numbers between 1 to $k-1$. If, to the contrary, we have $L_k>\eta L(f',S_k\cap U)$, then $\eta^{i_k-1}L_{k-1}>L(f',S_k\cap U)$ because $L_k/\eta=\eta^{i_k-1}L_{k-1}$. Hence, by using \beqref{eq:FQ} with $L:=\eta^{i_k-1}L_{k-1}$, we conclude that  \beqref{eq:L_k_ISTA} holds with $L$ instead of $L_k$, a contradiction to the minimality of $i_k$ unless $i_k=0$. But when $i_k=0$, we have $L_k=L_{k-1}$, hence, from the induction hypothesis and the fact that $L(f',S_{k-1}\cap U)\leq L(f',S_{k}\cap U)$ (since $L(f',S_k\cap U):=\sup\{\|f'(x)-f'(y)\|/\|x-y\|: x,y\in S_k\cap U, x\neq y\}$ and $S_{k-1}\subseteq S_k$), we have $L_k=L_{k-1}\leq \eta L(f',S_{k-1}\cap U)\leq \eta L(f',S_k\cap U)$, a contradiction to the assumption on  $L_k$. 
                 
Now we show that if we are in the backtracking step size rule and both $S_k=C$ for all $k\in \N$ and $L_1>\eta L(f',S_1\cap U)$, then $L_{k+1}=L_k$ for each $k\in\N$. Indeed, since $S_1=C$ and $L_{k+1}\geq L_k$ for each $k\in\N$ (as shown in Remark \bref{rem:BoundLk ISTA}), we have  $L_k>\eta L(f',C\cap U)$ for all $k\in\N$. If we do not have $L_{k+1}=L_k$ for each $k\in\N$, then $i_{k+1}>0$ for some $k\in\N$ and for this $k$ we have $\eta^{i_{k+1}-1}L_{k}=L_{k+1}/\eta>L(f',C\cap U)$. Thus \beqref{eq:FQ} with $L:=\eta^{i_{k+1}-1}L_{k}$ implies that \beqref{eq:L_k_ISTA} holds with $L$ instead of $L_k$, a contradiction to the minimality of $i_{k+1}$. Hence $L_{k+1}=L_k$ for each $k\in\N$.
\end{proof}

\begin{proof}[{\bf Proof of Lemma \bref{lem:BregmanLimit}}]

Assume to the contrary that there are at least two different weak cluster points $q_1:=w$-$\lim_{k\to\infty, k\in N_1}x_k$ and $q_2:=w$-$\lim_{k\to\infty, k\in N_2}x_k$ in $X$, where $N_1$ and $N_2$ are two infinite subsets of $\N$. By our assumption, $q_1,q_2\in U$, and hence, since $b$ satisfies the limiting difference property, we have
\begin{subequations}\label{eq:B(q_1,q_2)B(q_2,q_1)}
\begin{equation}
B(q_2,q_1)=\lim_{k\to\infty, k\in N_1}(B(q_2,x_k)-B(q_1,x_k))
\end{equation}
\quad\quad\quad\quad\quad\quad\quad and
\begin{equation}
B(q_1,q_2)=\lim_{k\to\infty, k\in N_2}(B(q_1,x_k)-B(q_2,x_k)).
\end{equation}
\end{subequations}
Since we assume that $L_1:=\lim_{k\to\infty}B(q_1,x_k)$ and $L_2:=\lim_{k\to\infty}B(q_2,x_k)$ exist and finite, we conclude from \beqref{eq:B(q_1,q_2)B(q_2,q_1)} that $B(q_2,q_1)=L_2-L_1=-(L_1-L_2)=-B(q_1,q_2)$. The assumptions on $b$ imply that $B$ is nonnegative (see, for example,  \cite[Proposition 4.13(III)]{ReemReichDePierro2019jour(BregmanEntropy)}), and hence $0\leq B(q_2,q_1)=-B(q_1,q_2)\leq 0$. Thus $B(q_1,q_2)=B(q_2,q_1)=0$. Since $b$ is strictly convex on $U$, for all $z_1\in \dom(b)$ and $z_2\in U$, we have $B(z_1,z_2)=0$ if and only if $z_1=z_2$ (see, for example, \cite[Proposition 4.13(III)]{ReemReichDePierro2019jour(BregmanEntropy)}). Thus $q_1=q_2$, a contradiction to the initial assumption. Thus all the weak cluster points of $(x_k)_{k=1}^{\infty}$ coincide. Since $(x_k)_{k=1}^{\infty}$ is a bounded sequence in a reflexive space, a well-known classical result implies that $(x_k)_{k=1}^{\infty}$ has at least one weak cluster point $q\in X$, which, by our assumption, is in $U$. We conclude from the above-mentioned discussion that all the weak cluster points of $(x_k)_{k=1}^{\infty}$  coincide with $q$. 

We claim that $(x_k)_{k=1}^{\infty}$ converges weakly to $q$. Indeed, otherwise there are a weak neighborhood $V$ of $q$ and a subsequence $(x_{k_j})_{j=1}^\infty$ of $(x_k)_{k=1}^\infty$ such that all its elements are not in $V$. But this subsequence  is a bounded sequence in a reflexive Banach space (since $(x_k)_{k=1}^{\infty}$ is bounded) and hence it has a convergent subsequence which converges weakly, as proved above, to $q$. In particular, infinitely  many elements of $(x_{k_j})_{j=1}^{\infty}$ are in $V$, a contradiction. 
\end{proof}

\begin{proof}[{\bf Proof of Lemma \bref{lem:SuperCoercive}}]
Let $\tilde{v}:X\to (-\infty,\infty]$ be defined by $\tilde{v}(x):=v(x)$ whenever $x\in S$ and $\tilde{v}(x):=\infty$ whenever $x\notin S$. Given $t\in\R$, we have $\{x\in X: \tilde{v}(x)\leq t\}=\{x\in S: v(x)\leq t\}$. Since $v$ is lower semicontinous on $S$, it follows that $\{x\in S: v(x)\leq t\}$ is closed in $S$, and since $S$ is closed in $X$, it follows that the sets$\{x\in S: v(x)\leq t\}$ and hence $\{x\in X: \tilde{v}(x)\leq t\}$ are closed in $X$. In other words, $\tilde{v}$ is lower semicontinuous on $X$. In addition, since $v$ is proper and convex on $S$, it is immediate to verify that $\tilde{v}$ is proper  and convex on $X$. Hence $\tilde{v}$ is minorized by  a continuous affine functional \cite[Lemma 6.12, p. 89]{VanTiel1984book}, that is, there exist a continuous linear functional $x^*:X\to\R$ and $\eta\in \R$ such that $\eta+\langle x^*,x\rangle\leq \tilde{v}(x)$ for all $x\in X$. Thus, by considering $x\in S$ and using the definition of $\|x^*\|$, the unboundedness of $S$ and the assumption on $w$, we obtain the required claim:
\begin{multline*}
\frac{u(x)}{\|x\|}=\frac{v(x)+w(x)}{\|x\|}\geq \frac{\eta+\langle x^*,x\rangle}{\|x\|}+\frac{w(x)}{\|x\|}
\geq \frac{w(x)}{\|x\|}+\frac{\eta}{\|x\|}-\|x^*\|_{\,\,\overrightarrow{\|x\|\to \infty, x\in S}}\,\,\infty.
\end{multline*}
\end{proof}

\begin{proof}[{\bf Proof of Lemma \bref{lem:LipschitzUpperBound}}]

Let $\phi:[0,1]\to\R$ be defined by $\phi(t):=f(y+t(x-y))$, $t\in [0,1]$. Since $f$ is continuously differentiable on $U$, it follows from the chain rule  that 
$\phi'$ exists, is continuous and $\phi'(t)=\langle f'(y+t(x-y)),x-y\rangle$. 
 By combining this fact with the fundamental theorem of calculus, with the 
  assumption  that $f'$ is Lipschitz continuous along the line segment $[x,y]$, 
   and with the monotonicity of the integral, we see that 
\begin{multline*}
f(x)=\phi(1)=\phi(0)+\int_0^1\phi'(t)dt=f(y)+\int_0^1\langle f'(y+t(x-y)),x-y\rangle dt\\
=f(y)+\int_0^1\langle f'(y),x-y\rangle dt+\int_0^1\langle f'(y+t(x-y))-f'(y),x-y\rangle dt\\
\leq f(y)+\int_0^1\langle f'(y),x-y\rangle dt+\int_0^1|\langle f'(y+t(x-y))-f'(y),x-y\rangle| dt\\
\leq f(y)+\langle f'(y),x-y\rangle+\int_0^1 L \|y+t(x-y)-y\|\|x-y\| dt\\
=f(y)+\langle f'(y),x-y\rangle+L\|x-y\|^2\int_0^1 tdt=
f(y)+\langle f'(y),x-y\rangle+\frac{1}{2}L\|x-y\|^2.
\end{multline*}
\end{proof}

\section*{Acknowledgements}
Part of the work of the first author was done when he was at the Institute of Mathematical and Computer Sciences (ICMC), University of S\~ao Paulo,  S\~ao Carlos, Brazil (2014--2016), and was supported by FAPESP 2013/19504-9. It is a pleasure for him to thank Alfredo Iusem and Jose Yunier Bello Cruz for helpful discussions regarding some of the references. The second author was partially supported by the Israel Science Foundation (Grants 389/12 and 820/17), by the Fund for the Promotion of Research at the Technion and by the Technion General Research Fund.  The third author thanks CNPq  grant 306030/2014-4 and FAPESP 2013/19504-9. All the authors wish to express their thanks to three referees for their feedback which helped to improve the presentation of the paper.

\bibliographystyle{amsplain}
\bibliography{biblio}

\end{document}